\newif\ifarxiv
\arxivtrue

\ifarxiv
    \documentclass[journal,twoside,web]{ieeecolor_arxiv}
\else
    \documentclass[journal,twoside,web]{ieeecolor}
\fi

\usepackage{generic}
\usepackage{cite}
\usepackage{amsmath,amssymb,amsfonts}
\usepackage{textcomp}
\def\BibTeX{{\rm B\kern-.05em{\sc i\kern-.025em b}\kern-.08em
    T\kern-.1667em\lower.7ex\hbox{E}\kern-.125emX}}

\usepackage[colorlinks=True,allcolors=blue]{hyperref}
\usepackage[margin=1in]{geometry}

\usepackage{wrapfig}

\usepackage{graphicx}
\usepackage{amsmath}
\usepackage{amssymb}
\usepackage{bm}
\usepackage{accents}

\newcommand{\uwidehat}[1]{%
  \mathpalette\douwidehat{#1}%
}
\makeatletter
\newcommand{\douwidehat}[2]{%
  \sbox0{$\m@th#1\widehat{\hphantom{#2}}$}%
  \sbox2{$\m@th#1x$}
  \sbox4{$\m@th#1#2$}
  \dimen0=\ht0
  \advance\dimen0 -.8\ht2
  \dimen2=\dp4
  \rlap{%
    \raisebox{\dimexpr\dimen0-\dimen2}{%
      \scalebox{1}[-1]{\box0}%
    }%
  }%
  {#2}%
}
\makeatother

\def\namedlabel#1#2{\begingroup
    #2%
    \def\@currentlabel{#2}%
    \phantomsection\label{#1}\endgroup
}
\makeatletter
\newcommand{\leqnomode}{\tagsleft@true\let\veqno\@@leqno}
\makeatother

\makeatletter
\newcommand{\customlabel}[2]{%
\protected@write\@auxout{}{\string \newlabel{#1}{{#2}{\thepage}{#2}{#1}{}}}%
\hypertarget{#1}{#2}
}
\makeatother

\newcommand\PMPODE[1]{\hyperref[PMPODE]{$\textbf{ODE}_{#1}$}\xspace}
\newcommand\PMPODElambda[1]{\hyperref[PMPODElambda]{$\textbf{ODE}_{#1}^\lambda$}\xspace}
\newcommand\PMPODElambdaunder[1]{\hyperref[PMPODElambda_under]{$\uwidehat{\textbf{ODE}_{#1}^\lambda}$}\xspace}
\newcommand\PMPODElambdaover[1]{\hyperref[PMPODElambda_over]{$\widehat{\textbf{ODE}_{#1}^\lambda}$}\xspace}
\newcommand\PMPODEepsilonfullrank[1]{\hyperref[PMPODEepsilon_full_rank]{$\textbf{ODE}_{#1}^\epsilon$}\xspace}
\newcommand\AssumXa{\hyperref[assumption:X0]{A4a}\xspace}
\newcommand\AssumXb{\hyperref[assumption:X0]{A4b}\xspace}

\newcommand\PMPODEg[1]{\hyperref[PMPODEg]{$\textbf{ODE}^g_{#1}$}\xspace}
\newcommand\OCPg[1]{\hyperref[OCPg]{$\textbf{OCP}^g_{#1}$}\xspace}

\newcommand\OCP[1]{\hyperref[OCP]{$\textbf{OCP}_{#1}$}\xspace}
\newcommand\BVP[1]{\hyperref[BVP]{$\textbf{BVP}_{#1}$}\xspace}
\newcommand\OCPMPC[1]{\hyperref[OCPMPC]{$\textbf{OCP}(#1)$}\xspace}

\usepackage{tabularx}
\newcolumntype{C}{>{\centering\arraybackslash}p{19mm}}
\newcolumntype{G}{>{\centering\arraybackslash}p{4mm}}
\newcolumntype{S}{>{\centering\arraybackslash\scriptsize}p{4mm}}

\usepackage[font=small,labelfont=bf]{caption}
\usepackage{color}
\usepackage[usenames,dvipsnames]{xcolor}

\usepackage{cleveref} 

\usepackage{xspace}

\let\labelindent\relax
\usepackage{enumitem} 

\usepackage{cite} 

\usepackage{amssymb}
\usepackage[plain]{algorithm}
\makeatletter
\renewcommand*{\ALG@name}{Alg.}
\makeatother
\usepackage[noend]{algpseudocode}

\usepackage{xpatch} 

\newtheorem{lemma}{Lemma}

\newtheorem{corollary}{Corollary}
 
\newtheorem{remark}{Remark}

\newcommand\blue[1]{\textcolor{blue}{#1}}

\newcommand\mydots{\hbox to 1em{.\hss.\hss.}}

\newcommand{\comp}{\mathsf{c}} 
\newcommand{\dd}{\textrm{d}} 
\newcommand{\hull}{\textrm{H}} 
\newcommand{\cS}{\mathcal{S}}

\newcommand{\Int}{\textrm{Int}}

\newcommand{\M}{\mathcal{M}}

\def\Inf{\operatornamewithlimits{inf\vphantom{p}}}

\newcommand{\C}{\mathcal{C}}
\newcommand{\X}{\mathcal{X}}
\newcommand{\cR}{\mathcal{R}}

\newcommand{\R}{\mathbb{R}}
\newcommand{\bN}{\mathbb{N}}
\newcommand{\W}{\mathcal{W}}

\newtheorem{preremark3}{Theorem}[section]

\usepackage{mdframed}
\usepackage{lipsum}
\newmdtheoremenv{theo}{Theorem}

\usepackage{multirow}

\newcommand\AverageSmallMatrix[1]{{%
  \footnotesize\arraycolsep=0.22\arraycolsep\ensuremath{\begin{bmatrix}#1\end{bmatrix}}}}

\newcommand\rev[1]{\textcolor{black}{#1}}

\usepackage[many]{tcolorbox}
\tcbuselibrary{breakable}
\makeatletter
\def\tcb@parbox@use@false{%
  \def\@parboxrestore{\linewidth\hsize\let\@parboxrestore=\tcb@parboxrestore}%
}
\makeatother

\begin{document}
\title{Convex Hulls of Reachable Sets}
\ifarxiv
    \author{Thomas Lew, 
    Riccardo Bonalli, 
    Marco Pavone
\thanks{The NASA University Leadership Initiative (grant \#80NSSC20M0163), Toyota Research Institute, and the Agence Nationale de la Recherche (grant ANR-22-CE48-0006) provided funds to assist the authors with their research, but this article solely reflects the opinions and conclusions of its authors. }
\thanks{Thomas Lew is with Toyota Research Institute, Los Altos, USA (e-mail: thomas.lew@tri.global). }
\thanks{Riccardo Bonalli is with the Laboratory of Signals \& Systems, Paris-Saclay University, CNRS, CentraleSup\'{e}lec, France (e-mail: riccardo.bonalli@l2s.centralesupelec.fr).}
\thanks{Marco Pavone is with 
the Department of Aeronautics \& Astronautics, Stanford University. Stanford, USA (e-mail: pavone@stanford.edu).}}
\fi

\maketitle
\ifarxiv
    \thispagestyle{empty}
    \pagestyle{empty}
\fi

\begin{abstract}
We study the convex hulls of reachable sets of 
nonlinear systems 
with bounded disturbances 
and uncertain initial conditions. 
Reachable sets play a critical role in control, but remain notoriously challenging to compute, and existing over-approximation tools tend to be conservative or computationally expensive. In this work, we characterize the convex hulls of reachable sets as the convex hulls of solutions of an ordinary differential equation with initial conditions on the sphere.
This finite-dimensional characterization unlocks an efficient sampling-based estimation algorithm to accurately over-approximate reachable sets. We also study the structure of the boundary of the reachable convex hulls and derive error bounds for the estimation algorithm. We give applications to neural feedback loop analysis and robust MPC.

\ifarxiv
    
    \vspace{2mm}
    
    \textbf{Code}: \href{https://github.com/StanfordASL/chreach}{https://github.com/StanfordASL/chreach}
\fi
\end{abstract}

\section{Introduction}
\label{sec:problem_formulation}
\IEEEPARstart{F}{orward} reachability analysis plays a critical role in control theory and robust controller design. Generally, it entails characterizing all states that a system can reach at any time in the future.  As such, reachability analysis allows certifying the performance of feedback loops under disturbances and designing controllers with robustness properties. In robust model predictive control (MPC) for instance, it is used to construct tubes around nominal state trajectories to ensure that constraints are satisfied in the presence of external disturbances.

In this work, we study the following reachability analysis problem. 
Let $n\in\bN$ be the state dimension, $f:\R\times\R^n\to\R^n$ and $g:\R\times\R^n\to\R^{n\times n}$ be functions for the dynamics, and $\W,\X_0\subset\R^n$ be bounded sets of disturbances and initial conditions. 
Given a time $T>0$ and an initial state $x^0\in\X_0$, we consider systems defined by the ordinary differential equation (ODE)
\begin{align}
\label{eq:ODE}
\dot{x}(t)&=f(t,x(t))+g(t,x(t))w(t),\quad t\in[0,T],
\\
x(0)&=x^0,
\nonumber
\end{align}
where the disturbances $w:[0,T]\to\W$ are assumed to be integrable ($w\in L^\infty([0,T],\W)$. %
Under standard smoothness assumptions (see Assumptions \ref{assumption:f}-\ref{assumption:W}), the ODE  \eqref{eq:ODE} has a unique solution, denoted by $x_{(w,x^0)}(\cdot)$. 
For any time $t\in[0,T]$, we define the reachable set 
\begin{align}\label{eq:Y_reachable_set}
\X_t=
\big\{
x_{(w,x^0)}(t): 
w\in L^\infty([0,T],\W),
\, 
x^0\in\X_0
\big\}
\end{align}
that characterizes all states that are reachable at time $t$ for some disturbance $w$ %
and initial state $x^0$. %

\begin{figure}[!t]
\includegraphics[width=1\linewidth]{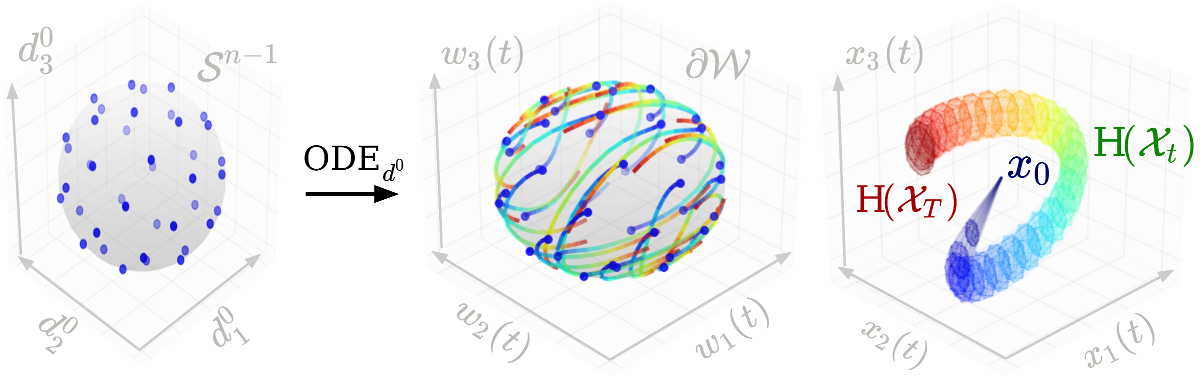}
\caption{The convex hulls $\hull(\X_t)$ of the reachable sets $\X_t$ can be computed by (a) integrating an augmented ODE  (\PMPODE{d^0}) for different directions $d^0$ on the sphere $\cS^{n-1}$, and (b) taking the convex hulls of the resulting extremal state trajectories $x_{d^0}$.}
\end{figure}

Reachability analysis of nonlinear dynamical systems is challenging. Indeed, from  \eqref{eq:Y_reachable_set}, computing the reachable sets seemingly requires evaluating an infinite number of state trajectories for all possible disturbances and initial conditions.\footnote{Each reachable set $\X_t$ is the image of an infinite-dimensional set. Indeed, by defining the maps $g_t(w,x^0)=x_{(w,x^0)}(t)\in\R^n$, each reachable set is expressed as $\X_t=g_t(L^\infty([0,T],\W)\times\X_0)$.} %
Due to the complexity of the problem, many existing tools seek \textit{convex over-approximations} of reachable sets. Yet, current methods tend to be conservative or computationally expensive, see Sections \ref{sec:related_work} and \ref{sec:numerical}. This motivates the study of properties of convex hulls of reachable sets that can simplify their estimation.

Our main contribution is a new characterization of 
the convex hulls of reachable sets of dynamical systems of the form \eqref{eq:ODE}, under smoothness assumptions of $f$, $g$, $\W$, and $\X_0$ (see Assumptions \ref{assumption:f}-\ref{assumption:X0}).  Specifically, denoting by $\hull(A)$ the convex hull of a set $A\subset\R^n$, we show that 
\begin{align}\label{eq:hull_Y_is_hull_F}
\hull(\X_t)=\hull(F(\cS^{n-1},t))
\ \text{ for all }\ t\in[0,T],
\end{align}
where $F(d^0,t)$ is the solution to an ODE with initial conditions $d^0$ on the sphere $\cS^{n-1}\subset\R^n$, %
see Theorem \ref{thm:hull_F} and \PMPODE{d^0}. Thus, the convex hulls of the reachable sets can now be computed as the convex hulls of solutions of an ODE for different initial conditions $d^0\in\cS^{n-1}$. Equation \eqref{eq:hull_Y_is_hull_F} represents a significantly simpler (finite-dimensional) characterization of the convex hulls. %

This result unlocks \rev{an} 
approach (Algorithm \ref{alg:1}) to efficiently estimate the convex hulls $\hull(\X_t)$ by integrating an ODE from a sample of initial conditions. 
This approach allows efficiently tackling  challenging problems such as analyzing the robustness of neural network controllers \rev{(see Section \ref{sec:numerical}). %
This characterization} also informs the design of a robust MPC controller (see Algorithm \ref{alg:mpc}) 
that we
demonstrate on a robust \rev{spacecraft control task}. 

This work extends preliminary results in \cite{LewBonalliEtAl2023} by:
\begin{itemize}[leftmargin=3mm]
\item considering time-varying disturbance-affine dynamics,
\item accounting for uncertain initial conditions,
\item studying the boundary of the convex hulls of reachable sets to obtain tighter error bounds (Section \ref{sec:error_bounds}),
\item analyzing problems with rectangular uncertainty sets (Section \ref{sec:rectangular}), and with disturbances that only affect a subset of directions of the statespace (Section \ref{sec:noninvertible}),
\item providing additional numerical results (Section \ref{sec:numerical}), 
\end{itemize}
The main characterization (see Theorem \ref{thm:hull_F} and \PMPODE{d^0}) also does not rely on the projection step from \cite{LewBonalliEtAl2023} anymore, simplifying the evaluation of solutions to \PMPODE{d^0}.

\subsubsection*{Outline}
In Section \ref{sec:related_work}, we review prior work. 
In Section \ref{sec:preliminary}, we introduce notations and preliminary results. 
In Sections \ref{sec:structure}-\ref{sec:structure_proof}, we state and derive our characterization result of the reachable convex hulls $\hull(\X_t)$ and propose an estimation algorithm (Algorithm \ref{alg:1}). 
In Sections \ref{sec:error_bounds}-\ref{sec:error_bound:proof}, we study the boundary of $\hull(\X_t)$ and derive error bounds for Algorithm \ref{alg:1}. 
We study problems with rectangular uncertainty sets and disturbances that only affect a subset of the statespace in Sections \ref{sec:rectangular}-\ref{sec:rectangular_proof} and  \ref{sec:noninvertible}, respectively. 
We provide numerical results in Section \ref{sec:numerical} and conclude in Section \ref{sec:conclusion}. The appendix contains additional details about theoretical and numerical results.

\section{Related work}\label{sec:related_work}
\subsubsection{Numerical methods}
The forward reachable sets of nonlinear systems are generally difficult to characterize. For this reason, many existing approaches seek \textit{convex} over-approximations of the reachable sets, e.g., %
represented as hyper-rectangles \cite{Meyer2021}, \rev{ellipsoids \cite{Kurzhanski2000}}, zonotopes \cite{Althoff2008}, or ellipsotopes \cite{Kousik2023}, 
\rev{see \cite{Althoff2021} for a recent survey that also reviews non-convex approximations.} 
Existing over-approximation methods include techniques based on conservative linearization \cite{Althoff2008}, differential inequalities \cite{Ramdani2009,Scott2013}, and Taylor models \cite{Berz1998,Chen2013}. %
In particular,  
systems with mixed-monotone \cite{meyer2019hscc,Coogan2015,Abate2022} or 
\ifarxiv
\rev{contracting \cite{Maidens2015,Fan2017,SinghMajumdarEtAl2017}} 
\fi
dynamics %
have been extensively studied, as these properties simplify the \rev{computation} of accurate over-approximations. 
To tackle smooth systems, a standard approach consists of linearizing the dynamics and bounding the Taylor remainder using smoothness properties of the dynamics \cite{Althoff2008,koller2018,Yu2013,Leeman2023,Althoff2021}. This method has been widely used in robust MPC but is known to be conservative \cite{LewPavone2020}, see also Section \ref{sec:numerical}. 

Methods that estimate the reachable sets from a sample of state \rev{trajectories \cite{Huang2012,Donz2007}} have recently found significant interest \cite{ThorpeL4DC2021,LewPavone2020,LewJansonEtAl2022}. However, the sample complexity of these methods increases with the number of uncertain variables. For systems with disturbances as in \eqref{eq:ODE}, the number of uncertain variables (and thus the approximation error) increases as the discretization is refined. Thus, naive sampling-based methods are not well-suited for reachability of systems with continuous-time disturbances, see also Section \ref{sec:numerical} for comparisons. 

\subsubsection{On geometry and optimal control}
The deep connection between geometry, reachability analysis, and optimal control is well-known \cite{Agrachev2004,BonnardChyba2003,Trelat2012}. It was previously used in \cite{Krener1989,Schttler2012} to characterize the true reachable sets of dynamical systems of dimensions $n\leq 4$ with scalar control inputs (control inputs in \cite{Krener1989} take the role of disturbances in \eqref{eq:ODE}). %
Our results also leverage geometric arguments and the Pontryagin Maximum Principle (PMP), but apply to a different class of dynamical systems with arbitrary state dimensionality $n$ and the same number of disturbances and states. With an appropriate relaxation scheme inspired from \cite{Silva2010}, these results can be approximately generalized to problems with a smaller number of disturbances than states, see Section \ref{sec:noninvertible}.  Importantly, by studying the convex hulls of the reachable sets, our results apply to arbitrarily-large times $T$ and sets $(\X_0,\W)$, and thus do not rely on a small-time assumption as in \cite{Krener1989} or on a set $\X_0$ small-enough as in \cite{Reiig2007}. In contrast, \cite{Krener1989} and \cite{Reiig2007} study the structure of the true reachable sets that may self-intersect for times $T$ too large, see Example \ref{example:selfintersect}.

Our derivations start with the idea of searching for boundary states that are the furthest in different directions %
(see \OCP{d}). This approach is standard in the setting with linear dynamics \rev{where  
reachable} sets are convex \cite{Pecsvaradi1971,Schttler2012},\rev{\cite[Chap.1.4]{Kurzhanski2014}}. However, in the nonlinear \rev{case, reachable} sets may be non-convex, and finding the extremal disturbance trajectories that generate boundary states %
requires solving optimal control problems (OCPs) or their corresponding boundary-value problems (BVPs) stemming from the PMP (see \BVP{d}). Such approaches were explored in \cite{Gornov2015} and \cite{Baier2009}, but remain computationally challenging. %
Our results show that under the right set of assumptions (see Assumptions \ref{assumption:f}-\ref{assumption:X0}), solving OCPs is not necessary and extremal trajectories take a simple form. %
The key is the additional idea of sampling initial values of the adjoint vector. %
Studying the convex hulls \rev{of  reachable} sets  unlocks arguments from convex geometry 
that allow proving 
the exactness of the approach. 

\section{Notations and preliminary results}\label{sec:preliminary}
\subsubsection{Notations}%
Let $a,b\in\R^n, \lambda\in\R$.  
We denote by $a^\top b=\sum_{i=1}^na_ib_i$ the Euclidean inner product, %
by $\|a\|=(\sum_{i=1}^na_i^2)^{1/2}$ the Euclidean norm, %
by $\|a\|_\lambda=(\sum_{i=1}^n|a_i|^\lambda)^{1/\lambda}$ the $\lambda$-norm %
with $\lambda\geq 1$,
by $a\odot b=(a_1b_1,\dots,a_nb_n)$ the elementwise product, %
$a^\lambda=(a_1^\lambda,\dots,a_n^\lambda)$, by $|a|=(|a_1|,\dots,|a_n|)$ the absolute value,    
by $I_n$ the identity matrix of size $n$, 
by $B(x,r)=\{y\in\R^n: \|y-x\|^2\leq r^2\}$ the closed ball of center $x$ and radius $r\geq 0$, and 
by $\cS^{n-1}=\{x\in\R^n: \|x\|^2=1\}$ the unit sphere. 
Given $A,B\subset\R^n$, %
we denote by 
$\Int(A)$, $\bar{A}$, $\partial A=\bar{A}\setminus\Int(A)$, and $A^\comp=\R^n\setminus A$ the interior, closure, boundary, and complement of $A$, 
by $d_A(x)=\inf_{a\in A} \|x-a\|$ the distance from $x$ to $A$, and by 
 \begin{equation}
 \label{eq:hausdorff_distance}
 d_H(A,B)
 =
 \max\left(
 \sup_{x\in A}
 d_B(x), 
 \sup_{y\in B}
 d_A(y)
 \right)
 \end{equation}
the Hausdorff distance between compact sets $A$ and $B$. 

\subsubsection{Convex geometry} \rev{A point in a set $A$  is said to be an extreme point if it is the endpoint of every segment in $A$ that contains it \cite{Grothendieck1973}.} We denote by $\hull(A)$ the convex hull of $A\subset\R^n$ and by    
$\text{Ext}(A)$ the set of extreme points of a compact set $A\subset\R^n$. 
The next result is standard. %
\begin{lemma}[Support hyperplane]\label{lem:hyperplane}
Let $C\subset\R^n$ be a closed and convex set and $x\in\partial C$. Then, there exists a support hyperplane $\{y\in\R^n:d^\top (y-x)=0\}$ defined by some $d\in\cS^{n-1}$ such that $d^\top x\geq d^\top y$ for all $y\in C$. 
\end{lemma}
The next result follows from the Krein-Milman theorem \cite{Grothendieck1973} and is also well-known, see \cite[Lemmas 6 and 7]{LewBonalliEtAl2023}.
\begin{lemma}\label{lem:extreme_points_hull}\label{lem:hull_is_hull_of_subset}
Let $A\subset\R^n$ be a compact set. Then, $\text{Ext}(\hull(A))\subseteq A$ and 
$\hull(A)=\hull(\partial A)=\hull(\partial\hull(A)\cap A)$.
\end{lemma}

\subsubsection{Differential geometry} 
Let $\M\subseteq\R^n$ be a $k$-dimensional submanifold. 
Equipped with the induced metric from the ambient Euclidean norm $\|\cdot\|$, $\M$ is a Riemannian submanifold. %
For any $x\in\M$, $T_x\M$ and $N_x\M$ denote the tangent and normal spaces of $\M$\cite{Lee2012}, respectively, which we view as linear subspaces of $\R^n$.  

Given a map $F:\R^m\to\R^n$ and $x,v,w\in\R^n$, $\dd F_x$ %
denotes the first-order differential of $F$ at $x$, with $\dd F_x(v)=\sum_{i=1}^m\frac{\partial F}{\partial x_i}(x)v_i$. %
 A differentiable map $F$ is a submersion if $\dd F_x:T_x\R^m\to T_{F(x)}\R^n$ is surjective for all $x\in\R^m$, and $F$ is a diffeomorphism if it is a bijection and its inverse is differentiable. %
For any $(t,x)\in[0,T]\times\R^n$, $\nabla f(t,x)\in\R^{n\times n}$ denotes the Jacobian matrix of $f(t,\cdot)$ at $x$ in Euclidean coordinates, and $\nabla g(t,x)\in\R^{n\times n\times n}$ is the $3$-tensor with entries $[\nabla g(t,x)]_{ijk}=(\partial g_{ij}(t,x) / \partial x_k)$. We define $\nabla g(t,x)v\in\R^{n\times n}$ with $[\nabla g(t,x)v]_{ik}=\sum_j[\nabla g(t,x)]_{ijk}v_j$  for any $v\in\R^m$.

\begin{figure}[!t]
\centering
\includegraphics[width=1\linewidth]{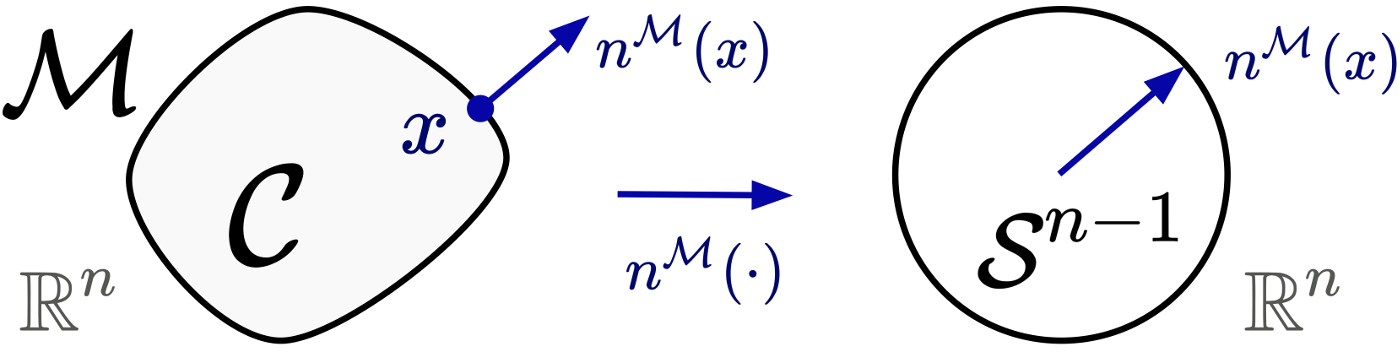}
\caption{Gauss map $n^\mathcal{M}:\mathcal{M}\to\cS^{n-1}$ of an ovaloid $\mathcal{M}=\partial\mathcal{C}$.}
\label{fig:gauss_map}
\vspace{-2mm}
\end{figure}
\subsubsection{Ovaloids and Gauss maps}
An $(n-1)$-dimensional submanifold $\M\subset\R^n$ is called a hypersurface. Let $\C\subset\R^n$ be such that $\M=\partial\C$ is a hypersurface. The \textit{Gauss map} of $\M$ is the map $n^{\M}:\M\to\cS^{n-1}$ defined such that $n^\M(x)$ %
is the unit-norm outward-pointing normal vector of $\M$ at $x\in\M$. 
For any $x\in\M$, the shape operator (or Weingarten map) is the linear map $S_x:T_x\M\to T_x\M$ defined by $S_x(v)=\nabla n(x) v$. The $(n-1)$ eigenvalues of the shape operator are called the \textit{principal curvatures} of $\M$. 
$\M$ is said to be an \textit{ovaloid} (or  of \textit{strictly positive curvature}), if all the principal curvatures of $\M$ are strictly positive. 
If $\M$ is an ovaloid, then 
 the Gauss map  $n^{\M}:\M\to\cS^{n-1}$
is a diffeomorphism \cite{Rauch1974}, 
 $\M$ is the boundary of a bounded strictly convex set $\mathcal{C}$ such that $\partial\mathcal{C}=\M$ \cite{Rauch1974}, and for  any $x\in\M$,
\begin{equation}\label{eq:gauss_map_ineq}
n^{\M}(x)^\top(x-y)\geq 0\ \ \text{for all } y\in\mathcal{C}.
\end{equation}

\begin{tcolorbox}[
title={\textit{Example \customlabel{example:gauss_maps}{1}(Gauss maps of common sets)}},
colbacktitle=gray!20, 
coltitle=black,
breakable,
boxrule=3pt,
colframe=gray!20, %
boxsep=0pt, before skip=0pt, after skip=0pt]
The Gauss maps of the boundaries of common ovaloids, such as the boundaries of balls and ellipsoids that are often used to represent $\W$ and $\X_0$, admit closed-form expressions.  Given $\bar{x}\in\R^n$ and $r>0$, the Gauss map of the boundary of the ball $B(\bar{x},r)$ and its inverse are
$$
n(x)=(x-\bar{x})/\|x-\bar{x}\|,
\ \ 
n^{-1}(d)=\bar{x}+rd.
$$
Given an ellipsoidal set $\mathcal{E}(\bar{x},Q)=\{x\in\R^n:(x-\bar{x})^\top Q^{-1}(x-\bar{x})\leq 1\}$  with  $Q\in\R^{n\times n}$ a positive definite matrix, the Gauss map of $\partial\mathcal{E}$ is 
$$
n(x)=\frac{Q^{-1}(x-\bar{x})}{\|Q^{-1}(x-\bar{x})\|},
\ \,  
n^{-1}(d)=\bar{x}+\frac{Qd}{\sqrt{d^\top Qd}}.
$$
For $\lambda>1$, the set 
$C_\lambda=\{x\in\R^n:\|x\|_\lambda^2\leq 1\}$ 
has boundary $\M=\partial C_\lambda=\{x\in\R^n:\|x\|_\lambda^2=1\}$. Its Gauss map and inverse Gauss map are
$$
n(x)=\frac{x\odot|x|^{\lambda-2}}{\||x|^{\lambda-1}\|},
\ 
n^{-1}(d)=\frac{d\odot |d|^{\frac{2-\lambda}{\lambda-1}}}{\||d|^{\frac{1}{\lambda-1}}\|_\lambda}.
$$
More generally, if $\M=\partial C$ is an ovaloid and is the level set of a smooth function $h:\R^n\to\R$,
$$\M=\{x\in\R^n:h(x)=1\},$$
then the Gauss map of $\M$ is \cite[Chapter 8]{Lee2018}
\begin{equation}\label{eq:gauss_map_level_set}
n(x)
=\nabla h(x) / \|\nabla h(x)\|.
\end{equation}
Moreover, under the additional assumption that  
$h\left(\nabla h^{-1}(n(x))\right)=1/\|\nabla h(x)\|^2$ for all $x\in\M$ (this assumption holds for the three examples
above, see Remark \ref{remark:gauss_map_level_set:inv:condition} in Appendix \ref{apdx:sec:gauss_map_level_set:inv} for details),
\begin{equation}\label{eq:gauss_map_level_set:inv}
n^{-1}(d)=\nabla h^{-1}\left(d\, / \sqrt{h(\nabla h^{-1}(d)}\right)
\end{equation}
(see Lemma \ref{lem:gauss_map_level_set:inv} in Appendix \ref{apdx:sec:gauss_map_level_set:inv} for details), from which one rederives the previous equations. 
%
%
\end{tcolorbox}

%
\section{The structure of $\hull(\X_t)$}\label{sec:structure}
Our results rely on the following four  assumptions.
\begin{tcolorbox}[
title={\textit{Assumption \customlabel{assumption:f}{1}($f$ and $g$ are smooth and integrable)}
},
coltitle=black,
boxrule=2pt,
colframe=blue!6, %
enhanced, colback=blue!2, boxsep=0pt, left=6pt, right=6pt]
For all $t\in[0,T]$, 
$f(t,\cdot)$ and $g(t,\cdot)$ are continuously differentiable, %
globally Lipschitz, and their Jacobians \scalebox{0.95}{$(\nabla f(t,\cdot),\nabla g(t,\cdot))$} are Lipschitz.
Moreover, for any $x\in\R^n$, $f(\cdot,x)$ and $g(\cdot,x)$ are integrable. %
\end{tcolorbox}
Assumption \ref{assumption:f} is a standard smoothness assumption \cite{Lorenz2005,Cannarsa2006} guaranteeing the existence and uniqueness of solutions to the ODE in \eqref{eq:ODE} and \PMPODE{d^0}. By multiplying $f$ and $g$ with a smooth cutoff function whose arbitrarily large support contains states of interest, the Lipschitzianity assumptions are always satisfied if $f,g\in C^2$. 
\begin{tcolorbox}[
title={\textit{Assumption \customlabel{assumption:g}{2}($g$ is invertible)}
\phantom{$a^1$}
},
coltitle=black,
boxrule=2pt,
colframe=blue!6, %
enhanced, colback=blue!2, boxsep=0pt, left=6pt, right=6pt]
$g(t,x)$ %
is invertible for all $(t,x)\in[0,T]\times\R^n$.
\end{tcolorbox}
Assumption \ref{assumption:g} %
does not hold for problems with fewer disturbances than states. It is relaxed in Section \ref{sec:noninvertible}.

\begin{tcolorbox}[
title={\textit{Assumption \customlabel{assumption:W}{3}($\W$ is convex with smooth boundary)}\phantom{${}^1$}
},
coltitle=black,
boxrule=2pt,
colframe=blue!6, %
enhanced, colback=blue!2, boxsep=0pt, left=6pt, right=6pt]
$\W$ is compact and its boundary $\partial\W$ is an ovaloid. 
Equivalently, $\partial\W$ is a smooth $(n-1)$-dimensional submanifold of strictly positive curvature.

\end{tcolorbox}
\begin{tcolorbox}[
title={\textit{Assumption \customlabel{assumption:X0}{4}($\X_0$ is either a singleton or is convex with smooth boundary)}
\phantom{$a^1$}
},
coltitle=black,
boxrule=2pt,
colframe=blue!6, %
enhanced, colback=blue!2, boxsep=0pt, left=6pt, right=6pt]
$\X_0$ is compact. In addition, either%
\\[-5mm]
\begin{itemize}[leftmargin=13mm]
\item[(A4a)] $\X_0=\{x^0\}$ for some $x^0\in\R^n$, or
\item[(A4b)] $\partial\X_0$ is an ovaloid. 
\end{itemize}
\end{tcolorbox}
\setcounter{assumption}{4}
Assumption \ref{assumption:W} and \AssumXb hold in particular if $\W$ and $\X_0$ are spheres or ellipsoids, which are commonly used in applications. 
These assumptions imply that $\W$ and $\X_0$ are strictly convex.  
They are relaxed in Section \ref{sec:rectangular}. 

Assuming that $\W$ is convex is standard to prove that the reachable sets are compact.%
\begin{lemma}[$\X_t$ is compact]\label{lem:Y_is_compact}
Assume that $f$, $g$, $\W$, and $\X_0$ satisfy Assumptions \ref{assumption:f}, \ref{assumption:W}, and \ref{assumption:X0}. Then, for any $t\in[0,T]$, the reachable set $\X_t=\eqref{eq:Y_reachable_set}$ is compact.
\end{lemma}

Lemma \ref{lem:Y_is_compact} is standard, see e.g. \cite{Trelat2023} (from Gr\"onwall's inequality, state trajectories are uniformly bounded thanks to Assumptions \ref{assumption:f}, \ref{assumption:W}, and \ref{assumption:X0}, so Lemma \ref{lem:Y_is_compact} follows from \cite[Theorem 7]{Trelat2023} with minor adaptations).

Thanks to Assumptions \ref{assumption:W} and \AssumXb, the \textit{Gauss maps}
\begin{equation}\label{eq:gauss_map}
n^{\partial\W}:\partial\W\to\cS^{n-1},
\quad n^{\partial\X_0}:\partial\X_0\to\cS^{n-1}
\end{equation} 
of $\partial\W$ and $\partial\X_0$ are diffeomorphisms, see Section \ref{sec:preliminary}. Recall that $n^{\partial\C}(x)$ is the unit outward-pointing normal vector of $\partial\C$ at $x\in\partial\C$ for $\C=\W$ and $\C=\X_0$, such that for any $w\in\partial\W$ and $v\in\W$,
\begin{equation}\label{eq:gauss_map_ineq:W}
n^{\partial\W}(w)^\top(w-v)\geq 0,
\end{equation}
and similarly for $\X_0$. 
If $\W=B(0,r)$ is a ball, then $n^{\partial\W}(w)=\frac{w}{\|w\|}$ and $(n^{\partial\W})^{-1}(d)=rd$, see Example \ref{example:gauss_maps}. %

Next, we state our main characterization result. %
Given any direction $d^0\in\cS^{n-1}$, we define the augmented ODE 
\begin{tcolorbox}[
title={\center\customlabel{PMPODE}{$\textbf{ODE}_{d^0}$}},
coltitle=black,
boxrule=2pt,
colframe=blue!6, %
enhanced, colback=blue!2, boxsep=1pt, left=-2pt, right=3pt]
\vspace{-3mm}
\begin{align}
\dot{x}(t)&=f(t,x(t))+g(t,x(t))w(t),
\quad  t\in[0,T],
\nonumber
\\
\dot{p}(t)&=
-(\nabla f(t,x(t))\,{+}\,\nabla g(t,x(t))w(t))^\top p(t),
\label{eq:pmpode:p_dot}
\\
\label{eq:pmpode:wt}
w(t)&=
(n^{\partial\W})^{-1}\left(\frac{g(t,x(t))^\top p(t)}{\|g(t,x(t))^\top p(t)\|}\right),
\\
x(0)&=
\begin{cases}
x^0
  &\hspace{-2mm}
  \text{if }\X_0=\{x^0\},
\\
(n^{\partial\X_0})^{-1}(d^0)
  &\hspace{-2mm}
  \text{if }\partial\X_0\text{ is an ovaloid},
\end{cases}
\nonumber
\\
p(0)&=d^0,
\nonumber
\end{align}
\end{tcolorbox}

\noindent which has a unique 
solution $(x,p)_{d^0}\in C([0,T],\R^{2n})$ thanks to Assumptions \ref{assumption:f}-\ref{assumption:X0}, from standard results on solutions of ODEs \cite{Lee2012}. %
Note that $g(t,x(t))^\top p(t)$ is always non-zero under Assumptions \ref{assumption:f}-\ref{assumption:X0} (see Lemma \ref{lem:singular_arcs}), so \eqref{eq:pmpode:wt} is well-defined. 
The next result characterizes the convex hulls of the reachable sets $\hull(\X_t)$ as the convex hull of solutions to \PMPODE{d^0} for all $d^0\in\cS^{n-1}$. 

\vspace{1mm}

\begin{tcolorbox}[
title={\textit{Theorem \customlabel{thm:hull_F}{1}(Convex hulls of reachable sets $\hull(\X_t)$)}\phantom{${}^1$}},
colbacktitle=blue!10, 
coltitle=black,
boxrule=3pt,
colframe=blue!10, %
enhanced, colback=blue!4, boxsep=0pt, left=6pt, right=6pt]
Assume that $f$, $g$, $\W$, and $\X_0$ satisfy Assumptions \ref{assumption:f}-\ref{assumption:X0}. 
Given $d^0\in\cS^{n-1}$, denote 
by $(x_{d^0},p_{d^0})$ the unique solution to \PMPODE{d^0}. 
Define the map
\begin{equation}\label{eq:F}
F:
\ 
\cS^{n-1}\times[0,T]\to\R^n: 
(d^0,t)\mapsto x_{d^0}(t). 
\end{equation}
Then, $\hull(\X_t)=\hull(F(\cS^{n-1},t))$ for all $t\in[0,T]$.
\end{tcolorbox}

\vspace{1mm}

Theorem \ref{thm:hull_F} states that integrating \PMPODE{d^0} for all values of $d^0\in\cS^{n-1}$ (i.e., evaluating $x_{d^0}(t)=F(d^0,t)$ for different directions $d^0$) is sufficient to recover the convex hulls of the reachable sets $\hull(\X_t)$. This characterization significantly simplifies the reachability analysis problem, which is now finite-dimensional and amounts to integrating an ODE from different initial conditions.

\begin{corollary}[Reachable tube]\label{cor:reachable_tube} Assume that $f$, $g$, $\W$, and $\X_0$ satisfy Assumptions \ref{assumption:f}-\ref{assumption:X0} and define $F=\eqref{eq:F}$ as in Theorem \ref{thm:hull_F}. 
Then, for all $t\in[0,T]$,  
\vspace{-1mm}
\begin{equation}\label{eq:reachable_tube}
\hull(\X_t)
=
\hull\Bigg(
\bigcup_{d^0\in\cS^{n-1}}
F(d^0,[0,T])_t
\Bigg),
\vspace{-1mm}
\end{equation}
where $F(d^0,[0,T])_t=F(d^0,t)$.
\end{corollary}
Corollary \ref{cor:reachable_tube} directly follows from Theorem \ref{thm:hull_F}. This result states that to recover  the reachable convex hull $\hull(\X_t)$ at any time $t$, %
it suffices to integrate \PMPODE{d^0} over $[0,T]$ only once for each initial direction $d^0\in\cS^{n-1}$. %
This result implies that all the information required to compute the entire convex reachable tube $\mathop{\bigcup}_{t\in[0,T]}\hull(\X_t)$ (e.g., to enforce constraints at all times for robust MPC, see Section \ref{sec:numerical:spacecraft}) is available after evaluating $\hull(\X_T)$. %

\begin{figure}[!t]
 \centering
\begin{tcolorbox}[
title={\textit{Alg. \customlabel{alg:1}{1}(Estimation of reachable set convex hulls)}\phantom{${}^1$}},
colbacktitle=ForestGreen!10, 
coltitle=black,
boxrule=3pt,
colframe=ForestGreen!10, %
colback=ForestGreen!4,
enhanced, boxsep=0pt, left=6pt, right=6pt]
 \vspace{-4mm}
 \begin{minipage}{1\linewidth}
 \begin{algorithm}[H]
 \textbf{Input}: $M$ initial directions $\{d^i\}_{i=1}^M\subset\cS^{n-1}$. %

 \textbf{Output}: Approximation of convex hulls $\hull(\X_t)$. %

 \begin{algorithmic}[1]
 \ForAll {$i=1,\dots,M$}
 \State \smash{$x^i\gets\textrm{Integrate}(\text{\PMPODE{d^i}})$}
\EndFor
 \State\Return $\hull\left(\{x^i(t), i=1,\dots,M\}\right),\ t\in[0,T]$. 
 \end{algorithmic}
 \end{algorithm}
 \end{minipage}
 \vspace{-3mm}
 \end{tcolorbox}
 \end{figure}
 
Theorem \ref{thm:hull_F} and Corollary \ref{cor:reachable_tube} justify using Algorithm \ref{alg:1} to reconstruct the convex hulls $\hull(\X_t)$. 
Error bounds for the approximation are derived in Section \ref{sec:error_bounds}.

\section{Proof of Theorem \ref{thm:hull_F}}%
\label{sec:structure_proof} 
\noindent We prove Theorem \ref{thm:hull_F} using convex geometry and optimal control. We first prove Theorem \ref{thm:hull_F} assuming that $\X_0=\{x^0\}$ (i.e., that Assumption \ref{assumption:X0} holds with \AssumXa) by searching for trajectories with endpoints $x(T)$ on the boundary of the reachable set $\X_t$. We then characterize the structure of such trajectories using the Pontryagin Maximum Principle (PMP) and conclude with an argument using convex geometry. Finally, we prove the case where $\partial\X_0$ is an ovaloid (i.e., Assumption \ref{assumption:X0} holds with \AssumXb). 
We discuss these results in Section \ref{sec:structure_proof:discussion}.
\subsection{Searching for extreme points of $\hull(\X_T)$}\label{sec:structure_proof:ocp}
Assume that $\X_0=\{x^0\}$. 
Let $d\in\cS^{n-1}$ be a search direction and define the optimal control problem (OCP)
\begin{align}
\textbf{OCP}_d:
\begin{cases}
\inf_{w\in L^\infty([0,T],\W)} \ 
  &-d^\top x(T)
\\[1mm]
\textrm{s.t.} \quad\, \  
&\hspace{-13mm}\dot{x}(t)=f(t,x(t))+g(t,x(t))w(t),
\\[1mm]
&\hspace{-13mm}t\in[0,T], \quad x(0)=x^0.
\end{cases}
\nonumber
\\[-6mm]
\label{OCP}
\end{align}
\OCP{d} is well-posed under Assumptions \ref{assumption:f} and \ref{assumption:W}, 
i.e., it admits at least one solution $w_d\in L^\infty([0,T],\W)$  (see, e.g., \cite[Theorem 9]{Trelat2023}, and note that $w(\cdot)=0$ is feasible). %
Intuitively, solving \OCP{d} gives a reachable state $x_d(T)\in\X_T$ that is the furthest in the direction $d$.

\subsection{Reformulating \OCP{d} using the PMP to reduce the search of solutions from $L^\infty([0,T],\W)$ to $\R^n$}\label{sec:structure_proof:bvp}
The Pontryagin Maximum Principle (PMP) \cite{Pontryagin1987,Agrachev2004,Trelat2012} gives necessary conditions of optimality for \OCP{d}. 
As the Hamiltonian of \OCP{d} is given by $
H(t,x,w,p)=p^\top (f(t,x)+g(t,x)w))
$, 
for any locally-optimal solution $(x_d,w_d)$ of \OCP{d}, there exists an absolutely-continuous function $p_d:[0,T]\to\R^n$, called the adjoint vector, such that for almost every $t\in[0,T]$, 
\begin{subequations}\label{eq:pmp}
    \begin{align}
\label{eq:pmp:pdot}
\hspace{-2mm}\dot{p}_d(t)&=-\left(\begin{aligned}
&\nabla f(t,x_d(t))\  +\qquad 
\\[-1mm]
&\qquad \nabla g(t,x_d(t))w_d(t)
\end{aligned}\right)^\top 
p_d(t)
\\
\label{eq:pmp:pT}
\hspace{-2mm}p_d(T)&=d
\\
\label{eq:pmp:wd}
\hspace{-2mm}w_d(t)&=\mathop{\arg\max}_{v\in\W}\ p_d(t)^\top g(t,x_d(t))v
\\
\label{eq:pmp:xdot}
\hspace{-2mm}\dot{x}_d(t)&=f(t,x_d(t))+g(t,x_d(t))w_d(t)
\\
\label{eq:pmp:x0}
\hspace{-2mm}x_d(0)&=x^0.
\end{align}
\end{subequations}
A tuple $(x_d, p_d, w_d)$ satisfying the above equations is called (Pontryagin) extremal for  \OCP{d}.  
These equations indicate that the adjoint vector is non-zero at all times.
\begin{lemma}[No singular arcs]\label{lem:singular_arcs} Assume that $(f,g,\W)$ satisfy Assumptions \ref{assumption:f}-\ref{assumption:W}. Let $(x_d, p_d, w_d)$ be an extremal for \OCP{d} with $d\in\cS^{n-1}$. Then, $p_d(t)\neq 0$  and  $p_d(t)^\top g(t,x_d(t))\neq 0$ for every $t\in[0,T]$.
\end{lemma}
\begin{proof}
By contradiction, $p_d(t)=0$ for some $t\in[0,T]$. Then, $0$ is the unique solution to the ODE $\dot{p}(s)=\eqref{eq:pmp:pdot}$ for $s\in[t,T]$ with $p(t)=0$. Thus, we obtain $d\mathop{=}\limits^\eqref{eq:pmp:pT}p_d(T)=0$, which is a contradiction. 
The result $p_d(t)^\top g(t,x_d(t))\neq 0$ for $t\in[0,T]$ follows from $p_d(t)\neq 0$ for $t\in[0,T]$ and Assumption \ref{assumption:g}. %
\end{proof}

Thanks to Lemma \ref{lem:singular_arcs} and Assumption \ref{assumption:W}, the maximality condition \eqref{eq:pmp:wd} can be simplified. 
  First, since $p_d(t)^\top g(t,x_d(t))\neq 0$ for all $t\in[0,T]$ thanks to Lemma \ref{lem:singular_arcs},  \eqref{eq:pmp:wd} is well-defined. 
 Second, since $\W$ is convex and $v \mapsto p_d(t)^\top g(t,x_d(t)) v$ is linear, searching for disturbances in $\partial\W$ suffices. %
Then, 
\begin{align}
w_d(t)
&=
\mathop{\arg\max}_{v\in\partial\W}\ p_d(t)^\top  g(t,x_d(t)) v
\nonumber
\\
&=
\mathop{\arg\max}_{v\in\partial\W}\ \frac{p_d(t)^\top g(t,x_d(t)) }{\|p_d(t)^\top g(t,x_d(t)) \|} v
\nonumber
\\
&=
\left(n^{\partial\W}\right)^{-1}\left(\frac{p_d(t)^\top g(t,x_d(t)) }{\|p_d(t)^\top g(t,x_d(t)) \|}\right)
\label{eq:wstar}
\end{align}
where $n^{\partial\W}:\partial\W\to\cS^{n-1}$ is the Gauss map in \eqref{eq:gauss_map}, which is a diffeomorphism since $\partial\W$ is an ovaloid \cite{Rauch1974} by  Assumption \ref{assumption:W}. The last equality in \eqref{eq:wstar} follows from \eqref{eq:gauss_map_ineq:W} (note that $\eqref{eq:gauss_map_ineq:W}=0$ if and only if $v=w(t)$, due to the strict convexity of $\partial\W$).  
Thus, by combining \eqref{eq:pmp} and \eqref{eq:wstar}, we obtain that candidate optimal solutions of \OCP{d} must solve the boundary-value problem (BVP) 
\begin{align}\label{BVP}
\textbf{BVP}_d:  
\begin{cases}
\dot{x}_d(t)=\eqref{eq:pmp:xdot}, 
\quad 
\dot{p}_d(t)=\eqref{eq:pmp:pdot},
\\[0mm]
w_d(t)=\eqref{eq:wstar}, \quad \ 
t\in[0,T],
\\[0mm]
(x_d(0),p_d(T))=(x^0,d).
\end{cases}
\end{align} 
With \BVP{d}, we reduced the search of solutions to \OCP{d} from $w\in L^\infty([0,T],\W)$ to $p_d(0)\in\R^n$.

\subsection{Reformulating \BVP{d} with knowledge of $\frac{p_d(0)}{\|p_d(0)\|}$}\label{sec:structure_proof:guessing}

\BVP{d} and \eqref{eq:wstar} indicate that extremal state trajectories $x_d$ follow dynamics that only depend on $\frac{p_d}{\|p_d\|}$. The next result shows that extremal trajectories are independent of the norm of $p_d(0)$. Thus, it suffices to search over $\cS^{n-1}$ to retrieve extremal trajectories $x_d$.

\begin{lemma}[Extremals of \OCP{} are identified by $\frac{p(0)}{\|p(0)\|}$]
\label{lem:qt_solves_ode}
Assume that $(f,g,\W)$ satisfy Assumptions \ref{assumption:f}-\ref{assumption:W}. Let $d\in\cS^{\scalebox{0.7}{$n{-}1$}}$ and $(x_d, p_d, w_d)$ be an extremal for \OCP{d}. 
Then, there exist a direction $d^0\in\cS^{n-1}$ and an adjoint trajectory $\tilde{p}:[0,T]\to\R^n$ with $\tilde{p}(0)=d^0$ such that  
$(x_d, \tilde{p}, w_d)$ solves \PMPODE{d^0} with $\X_0=\{x^0\}$.
\end{lemma} 
\begin{proof}
First, for all $t\in[0,T]$, we define 
\begin{equation}\label{eq:q}
q_d(t)=\frac{p_d(t)}{\|p_d(t)\|},
\end{equation}
which is well-defined for all $t\in[0,T]$ thanks to Lemma \ref{lem:singular_arcs} and such that $q_d(0)=\frac{p_d(0)}{\|p_d(0)\|}\in\cS^{n-1}$. We define 
\begin{equation}\label{eq:d0_q0}
d^0=q_d(0).
\end{equation}
and write $(x_t,p_t,w_t,q_t)=(x_d(t),p_d(t),w_d(t),q_d(t))$ for conciseness.  
As $(x, p, w)$ is an extremal for \OCP{d}, 
\begin{subequations}\label{eq:ode_wxq}
\begin{align}
&w_t
\mathop{=}\limits^\eqref{eq:wstar}
\left(n^{\partial\W}\right)^{-1}\left(\frac{q_t^\top g(t,x_t) }{\|q_t^\top g(t,x_t) \|}\right),
\label{eq:wstar_q}
\\
&\dot{x}_t
\mathop{=}\limits^\eqref{eq:pmp:xdot}
f(t,x_t)+g(t,x_t)w_t,
\\
&\dot{q}_t
=
\frac{\dot{p}_t}{\|p_t\|} -p_t\frac{p_t^\top\dot{p}_t}{\|p_t\|^{3}}
=
\left(I_n-\frac{p_tp_t^\top}{\|p_t\|^2}\right)
\frac{\dot{p}_t}{\|p_t\|}
\nonumber
\\
&\hspace{-1mm}\mathop{=}\limits^\eqref{eq:pmp:pdot}
{-}\left(I_n{-}q_tq_t^\top\right)
(\nabla f(t,x_t){+}\nabla g(t,x_t)w_t)^\top q_t,
\\
&q_0\mathop{=}\limits^\eqref{eq:d0_q0}d^0.
\end{align}
\end{subequations}
Next, let $(\tilde{x},\tilde{p},\tilde{w})$ be the solution to \PMPODE{d^0} and define $\tilde{q}=\frac{\tilde{p}}{\|\tilde{p}\|}$. We claim that $(\tilde{x},\tilde{w})=(x,w)$. Indeed,
\begin{subequations}\label{eq:ode_wxq_tilde}
\begin{align}
&\tilde{w}_t
=
\left(n^{\partial\W}\right)^{-1}\left(\frac{\tilde{q}_t^\top g(t,\tilde{x}_t) }{\|\tilde{q}_t^\top g(t,\tilde{x}_t) \|}\right),
\\
&\dot{\tilde{x}}_t
=
f(t,\tilde{x}_t)+g(t,\tilde{x}_t)\tilde{w}_t,
\\
&\dot{\tilde{q}}_t
=
-(I_n{-}\tilde{q}_t\tilde{q}_t^\top)
(\nabla f(t,\tilde{x}_t){+}\nabla g(t,\tilde{x}_t)\tilde{w}_t)^\top \tilde{q}_t,
\\
&\tilde{q}_0=d^0.
\end{align}
\end{subequations}
By uniqueness of solutions to ODEs, from \eqref{eq:ode_wxq} and \eqref{eq:ode_wxq_tilde}, we conclude that $(\tilde{x},\tilde{q},\tilde{w})=(x,q,w)$, and in particular that  $(\tilde{x},\tilde{w})=(x,w)$. The conclusion follows.
\end{proof}

\subsection{Concluding the proof of Theorem \ref{thm:hull_F}}%
\label{sec:structure_proof:thm1_A4a}
Lemma \ref{lem:qt_solves_ode} yields the following key result. 
\begin{lemma}\label{lem:partial_HY_subset_F_subset_Y}
Assume that $f$, $g$, and $\W$ satisfy Assumptions \ref{assumption:f}-\ref{assumption:W}, $\X_0$ satisfies Assumption \ref{assumption:X0} with \AssumXa ($\X_0=\{x^0\}$ is a singleton), and define $F=\eqref{eq:F}$ as in Theorem \ref{thm:hull_F}. 
Then, for all $t\in[0,T]$, $F(\cdot,t)$ is smooth and
\begin{align}\label{eq:partial_HY_subset_F_subset_Y}
\left(\partial\hull(\X_t)\cap\X_t\right)\subseteq F(\cS^{n-1},t)\subseteq\X_t.
\end{align}
\end{lemma}
\begin{proof}
Without loss of generality, we prove the result for $t=T$. %
The proof can be extended to $t\in[0,T)$ by defining \OCP{d}\hspace{-2.3mm}\blue{${}^t$} to maximize $d^\top x(t)$, which results in the same expressions for \PMPODE{} and $F$. %

First, $F(\cdot,T)$ is smooth since it is the solution to an ODE with smooth coefficients. %

Second, $F(\cS^{n-1},T)\subseteq\X_T$ by definition. To show the other inclusion, let $y\in\partial\hull(\X_T)\cap\X_T$. 
Since $y\in\X_T$, there exists some $w\in L^\infty([0,T],\W)$ such that $y=x_w(T)$ where $x_w$ solves the ODE in \eqref{eq:ODE}. Then, since $x_w(T)=y\in\partial\hull(\X_T)$, by  the convexity of $\hull(\X_T)$ and Lemma \ref{lem:hyperplane}, %
$x_w(T)$ maximizes the function $\tilde{w}\mapsto d^\top x_{\tilde{w}}(T)$ over $\tilde{w}\in L^\infty([0,T],\W)$ for some $d\in\cS^{n-1}$, i.e., $(x_w,w)$ solves \OCP{d}. %
 Thus, by Lemma \ref{lem:qt_solves_ode}, $x_w$ solves \PMPODE{d^0} for some $d^0\in\cS^{n-1}$. 
We obtain $y=x_w(T)=F(d^0,T)\in F(\cS^{n-1},T)$. 
\end{proof}

Theorem \ref{thm:hull_F} (for the case where $\X_0=\{x^0\}$) almost immediately follows from Lemmas  \ref{lem:hull_is_hull_of_subset} and  \ref{lem:partial_HY_subset_F_subset_Y}.  To prove the case where $\partial\X_0$ is an ovaloid, we define a dynamical system with the same reachable sets but with a fixed initial state and conclude with the previous result. %

\textit{Proof of Theorem \ref{thm:hull_F}  if $\X_0=\{x^0\}$ (Assumption \ref{assumption:X0} holds with \AssumXa):} 
For any $t\in[0,T]$, $\hull(\X_t)=\hull(F(\cS^{n-1},t))$  %
follows from taking the convex hull on both sides of \eqref{eq:partial_HY_subset_F_subset_Y} and using $\hull\left(\partial\hull(\X_t)\cap\X_t\right)=\hull(\X_t)$ (Lemma \ref{lem:hull_is_hull_of_subset})  since $\X_t$ is compact (Lemma \ref{lem:Y_is_compact}). 
\hfill $\blacksquare$ %

\textit{Proof of Theorem \ref{thm:hull_F}  if $\partial\X_0$ is an ovaloid  (Assumption \ref{assumption:X0} holds with \AssumXb):} 
We define the new ODE
\begin{align}
\label{eq:ODE_extended}
&\hspace{-2mm}\dot{\tilde{x}}(t)=\begin{cases}
v(t)&\hspace{-2mm}\text{if }t\in[-1,0]
\\
f(t,\tilde{x}(t))+g(t,\tilde{x}(t))w(t) &\hspace{-2mm}\text{if }t\in[0,T]
\end{cases}
\\
&\hspace{-2mm}\tilde{x}(-1)=0,
\nonumber
\end{align}
where $w\in L^\infty([0,T],\W)$ and $v\in L^\infty([-1,0],\X_0)$. Under Assumptions  \ref{assumption:f}-\ref{assumption:X0}, this ODE has a unique solution, denoted by $\tilde{x}_{(w,v)}(\cdot)$. 
We define the reachable sets 
$
\tilde{\X}_t
=
\left\lbrace
\tilde{x}_{(w,v)}(t): 
w\in L^\infty([0,T],\W),
\, 
v\in L^\infty([-1,0],\X_0)
\right\rbrace
$ 
for $t\in[-1,T]$. By definition, %
$\tilde{\X}_0=\X_0$ and 
\begin{equation}\label{eq:tilde_X_is_X}
\tilde{\X}_t=\X_t\text{ for all }t\in[0,T].
\end{equation}

\ifarxiv
    \begin{table}[t]
    \begin{center}
    \caption{Problems used to prove Theorem \ref{thm:hull_F}.%
    }\label{table:problems}
    \begin{tabular}{ |l|c|c|c| } 
     \hline
     Problem & unknown variables & number of variables
     \\
     \hline
     \hline
     \OCP{d} & $w\in L^\infty([0,T],\W)$ & infinite 
     \\
     \hline
     \BVP{d} & $p(0)\in\R^n$ & $n$
     \\
     \hline
     \PMPODE{d^0} & None & $0$
     \\
    \hline
    \end{tabular}
    \end{center}
    \vspace{-4mm}
    \end{table}

\fi

Given any $d^0\in\cS^{n-1}$, we define the ODE 
\begin{align}\label{PMPODE_extended}
\begin{cases}
\dot{\tilde{x}}(t)=\eqref{eq:ODE_extended},
\qquad\qquad 
\left(\widetilde{\textbf{ODE}}_{d^0}\right)
  &\hspace{-2.5mm}\scalebox{0.9}{$t\in[-1,T]$}
\\
\dot{\tilde{p}}(t)=
0,
  &\hspace{-2.5mm}\scalebox{0.9}{$t\in[-1,0]$}
\\
\tilde{v}(t)=
(n^{\partial\X_0})^{-1}\left(\tilde{p}(t)/\|\tilde{p}(t)\|\right),
  &\hspace{-2.5mm}\scalebox{0.9}{$t\in[-1,0]$}
\\
\dot{\tilde{p}}(t)=-\left(\begin{aligned}
&\nabla f(t,\tilde{x}(t))\  +\qquad 
\\
&\qquad\quad \nabla g(t,\tilde{x}(t))\tilde{w}(t)
\end{aligned}\right)^\top 
\tilde{p}(t),
&\hspace{-1mm}\scalebox{0.9}{$t\in[0,T]$}
\\
\tilde{w}(t)=
(n^{\partial\W})^{-1}\left(\frac{g(t,\tilde{x}(t))^\top \tilde{p}(t)}{\|g(t,\tilde{x}(t))^\top \tilde{p}(t)\|}\right),
  &\hspace{-1mm}\scalebox{0.9}{$t\in[0,T]$}
\\
\tilde{x}(-1)=0,
\\
\tilde{p}(-1)=d^0,
\end{cases}
\nonumber
\\[-7mm]
\end{align} 
which has a unique 
solution $(\tilde{x},\tilde{p})_{d^0}\in C([-1,T],\R^{2n})$ thanks to Assumptions \ref{assumption:f}-\ref{assumption:X0}, and the map
\begin{equation}\label{eq:F_extended}
\tilde{F}:
\ 
\cS^{n-1}\times[-1,T]\to\R^n: 
(d^0,t)\mapsto \tilde{x}_{d^0}(t). 
\end{equation}
Theorem \ref{thm:hull_F} (with fixed initial condition $\tilde{x}(-1)\,{=}\,0$) gives
$$
\hull(\X_t)\mathop{=}\limits^{\eqref{eq:tilde_X_is_X}}\hull(\tilde{\X}_t)=\hull(\tilde{F}(\cS^{n-1},t))
\text{ for all }t\in[-1,T].
$$
Next, from \hyperref[PMPODE_extended]{
$\widetilde{\textbf{ODE}}_{d^0}$}, $\tilde{p}(t)=d^0$ for all $t\in[-1,0]$, so $\tilde{v}(t)=(n^{\partial\X_0})^{-1}(d^0)$ for all $t\in[-1,0]$, and
\begin{equation}\label{eq:xptilde_at_t=0}
\left(\tilde{x}(0),\tilde{p}(0)\right)=\left((n^{\partial\X_0})^{-1}(d^0),d^0\right).
\end{equation}
Thus,  \hyperref[PMPODE_extended]{
$\widetilde{\textbf{ODE}}_{d^0}$} restricted to $t\in[0,T]$ from \eqref{eq:xptilde_at_t=0} is exactly \PMPODE{d^0}, which concludes the proof of Theorem \ref{thm:hull_F}. 
\hfill $\blacksquare$ %

\subsection{Discussion and insights}\label{sec:structure_proof:discussion}


In Table \ref{table:problems}, we summarize the different problems used to derive \PMPODE{d^0} and ultimately prove Theorem \ref{thm:hull_F}. 

\makeatletter
\xpatchcmd{\algorithmic}
  {\ALG@tlm\z@}{\leftmargin\z@\ALG@tlm\z@}
  {}{}
\makeatother

\begin{wrapfigure}{R}{0.63\linewidth}
\vspace{-5mm}
\begin{tcolorbox}[
title={\textit{Alg. \customlabel{alg:2}{2}(optimal control scheme)}},
colbacktitle=ForestGreen!10, 
coltitle=black,
boxrule=3pt,
colframe=ForestGreen!10, %
colback=ForestGreen!4,
enhanced, boxsep=0pt, left=6pt, right=6pt]
 \begin{minipage}{1\linewidth}
 \begin{algorithm}[H]
 \vspace{-5mm}
 \centering
 {\small
 \begin{algorithmic}
 \ForAll {$d\in\cS^{n-1}$}
 \State $x_d\gets\textrm{Solve}($\OCP{d} or \BVP{d}$\hspace{-1mm})$
\EndFor
 \State\Return $\hull\left(x_d(T), d\in\cS^{n-1}\right)$ 
 \end{algorithmic}
 }%
 \end{algorithm}
 \end{minipage}
 \vspace{-3mm}
 \end{tcolorbox}
 \vspace{-6mm}
 \end{wrapfigure}

 At first sight, \OCP{d} and \BVP{d} suggest using  Algorithm \ref{alg:2} to reconstruct the convex hull of the reachable set $\hull(\X_T)$ (similar ideas are investigated in \cite{Gornov2015} and in \cite{Baier2009}). 
However, this procedure can be computationally expensive. Also, \OCP{d} is generally non-convex, so Algorithm \ref{alg:2}  could be prone to local minima and under-estimating the reachable sets. Thus, Algorithm \ref{alg:2} may be  unsuitable for applications that require efficient reachable set over-approximations. Carrying on the analysis using samples of $p_d(0)\in\cS^{n-1}$ and observing that the norm of $p_d(t)$ does not play a role in the problem (Section \ref{sec:structure_proof:guessing}) is key to our result.

Lemma \ref{lem:qt_solves_ode} implies that extremal trajectories are completely specified by the initial value of the adjoint vector $p_d(0)\in\cS^{n-1}$. Thus, given $p_d(0)=d^0$, we can integrate \PMPODE{d^0} to obtain the corresponding reachable extremal states $x_d(t)$, independently of the search direction $d$ (as $d$  is implicitly encoded in $p_d(0)$). This observation is the key insight behind Algorithm \ref{alg:1}, that consists of integrating \PMPODE{d^0} for different values of $d^0\in\cS^{n-1}$ to recover the convex hulls of the reachable sets. %
\rev{To gain further intuition with the linear case, see  Appendix \ref{apdx:linear_case}.}

\section{The boundary structure of $\hull(\X_t)$, geometric estimation, and error bounds}\label{sec:error_bounds}
How accurate are the estimates returned by Algorithm \ref{alg:1}, which approximates the reachable convex hulls $\hull(\X_t)$ with the convex hulls of a finite number of state trajectories?  
First, we show that the boundaries of the convex hulls of reachable sets are  smooth submanifolds under Assumptions \ref{assumption:f}-\ref{assumption:X0} (Lemma \ref{lem:partial_hull_reachable_submanifold}). This smooth boundary structure implies tight error bounds for convex-hull sampling-based estimators (Theorem \ref{thm:error_bound}). Indeed, error bounds of sample-based approximations typically rely on smoothness properties of the functions of interest and on sufficient coverage of the samples, as shown below.

\begin{corollary}[Naive error bound]\label{cor:error_bound}
Assume that $f$, $g$, $\W$, and $\X_0$ satisfy Assumptions \ref{assumption:f}-\ref{assumption:X0}. 
Let $\delta>0$, $Z_\delta\subset\cS^{n-1}$ be a $\delta$-cover of $\cS^{n-1}$ (i.e., $\cS^{n-1}\subset Z_\delta+B(0,\delta)$), and define $F=\eqref{eq:F}$. 
Then
\begin{equation}\label{eq:error_bound_naive}
\hspace{-0.5pt}
d_H(\hull(\X_t),\hull(F(Z_\delta,t)))\leq\bar{L}_t\delta\, 
\text{ for all }t\in[0,T],
\hspace{-1mm}
\end{equation}
where $\bar{L}_t$ denotes the Lipschitz 
\ifarxiv
    constant\footnote{$F(\cdot, t)$ is Lipschitz since it is differentiable and $\cS^{n-1}$ is compact.} 
\fi
of $F(\cdot,t)$. 
\end{corollary}
Corollary \ref{cor:error_bound} follows from Theorem \ref{thm:hull_F} using a standard covering argument, see \cite[Lemma 4.2]{LewBonalliJansonPavone2023}.
It implies that given a sufficiently-dense sample $Z_\delta=\{d^i\}_{i=1}^M$ that $\delta$-covers $\cS^{n-1}$, padding the set estimates from Algorithm \ref{alg:1} by $\epsilon_t=\bar{L}_t\delta$ suffices to obtain over-approximations of the reachable sets $\X_t$. 
However, %
Corollary \ref{cor:error_bound} does not fully exploit the smoothness of the sets of interest. %

\begin{lemma}[$\partial\hull(\X_t)$ is smooth]\label{lem:partial_hull_reachable_submanifold}
Assume that $f$, $g$, $\W$, and $\X_0$ satisfy Assumptions \ref{assumption:f}-\ref{assumption:X0}. Then,  
$\partial\hull(\X_t)$ is an $(n-1)$-dimensional submanifold of $\R^n$ for any $t>0$.
\end{lemma}

The proof of Lemma \ref{lem:partial_hull_reachable_submanifold} uses the structure of extremal trajectories and is in Section \ref{sec:error_bound:proof}.\footnote{A version of Lemma \ref{lem:partial_hull_reachable_submanifold} quantifying the smoothness of the boundary $\partial\hull(\X_t)$ can be derived by combining the interior smoothness properties of reachable sets  in \cite{Lorenz2005,Cannarsa2006}  
and of convex hulls in \cite{LewBonalliJansonPavone2023}.}
Thanks %
to the smoothness of the dynamics, of the input set $\cS^{n-1}$ for the map $F$, and of the reachable convex hull boundary $\partial\hull(\X_t)$ (see \cite{LewBonalliJansonPavone2023} for a quantitative definition of smoothness of sets), Algorithm \ref{alg:1} admits the following error bounds.

\begin{tcolorbox}[
title={\textit{Theorem \customlabel{thm:error_bound}{2}(Estimation error)}},colbacktitle=blue!10, 
coltitle=black,
boxrule=3pt,
colframe=blue!10, %
enhanced, colback=blue!4, boxsep=0pt, left=6pt, right=6pt]
Assume that $f$, $g$, $\W$, and $\X_0$ satisfy Assumptions \ref{assumption:f}-\ref{assumption:X0}. 
Let $\delta>0$, $Z_\delta\subset\cS^{n-1}$ be a $\delta$-cover of $\cS^{n-1}$, and define $F=\eqref{eq:F}$. 
Then, 
\begin{equation}\label{eq:error_bound}
d_H(\hull(\X_t),\hull(F(Z_\delta,t)))\leq
\left(\frac{\bar{L}_t+\bar{H}_t}{2}\right)\delta^2
\end{equation}
for all $t\in[0,T]$, where $(\bar{L}_t,\bar{H}_t)$ are the Lipschitz constants of $(F^t,\dd F^t)$, where $F^t(\cdot)=F(\cdot,t)$. 
\end{tcolorbox}
The error bound in \eqref{eq:error_bound} is quadratic in $\delta$. 
It is thus tighter than the naive error bound in Corollary \ref{cor:error_bound} for smaller values of $\delta$ (i.e., for sufficiently-many samples of $d^0$ so that $\delta\leq (2\bar{L}_t)/(\bar{L}_t+\bar{H}_t)$). 

\rev{According to Theorem \ref{thm:error_bound}, the sample complexity of Algorithm \ref{alg:1} is exponential in the dimension of the sample space $\cS^{n-1}$, as the minimum number of samples to $\delta$-cover a compact set scales exponentially with the dimension of the set \cite[Eq.(5.9)]{wainwright_2019}, so the performance of Algorithm \ref{alg:1} may degrade as the state dimension $n$ increases. This limitation is shared by other algorithms and  is known as the curse of dimensionality. Nevertheless, thanks to Theorem \ref{thm:hull_F}, the sample space is only of dimension $(n-1)$ as opposed to an infinite-dimensional space of disturbances, so we expect better performance than if  naively sampling disturbances, see Section \ref{sec:numerical}.}
\begin{remark}[Proving Theorem \ref{thm:error_bound}]\label{remark:deriving_error_bounds}
The proof of Theorem \ref{thm:error_bound} relies on Theorem \ref{thm:hull_F}. It takes inspiration from \cite[Theorem 1.1]{LewBonalliJansonPavone2023}, but requires new analysis due to several difficulties. First, the map $F(\cdot,t)$ is not a diffeomorphism onto its image: $\cS^{n-1}$ is an $(n-1)$-dimensional submanifold of $\R^n$, but the set $F(\cS^{n-1},t)$ may self-intersect and is thus not a submanifold of $\R^n$, see Example \ref{example:selfintersect}. 
$F$ is neither a submersion: $\dd F(d^0,t):T_{d^0}\cS^{n-1}\to T_{F(d^0,t)}\R^n$ cannot be surjective  since $\cS^{n-1}$ is only $(n-1)$-dimensional. %
To prove Theorem \ref{thm:error_bound}, we exploit properties of solutions to \PMPODE{d^0} and of $F$. 
Theorem \ref{thm:error_bound} relies on Lemma \ref{lem:partial_hull_reachable_submanifold}, whose proof relies on the PMP and the structure of extremal trajectories from the PMP coupled with properties of convex sets.
\end{remark}

\begin{figure}[t]
\begin{tcolorbox}[
title={\textit{Example \customlabel{example:selfintersect}{2}(Intersections)\phantom{${}^1$}}},colbacktitle=gray!14, 
coltitle=black,
boxrule=3pt,
colframe=gray!14, %
colback=gray!6,
boxsep=0pt, before skip=0pt, after skip=0pt]
\begin{minipage}{0.62\linewidth}
    \centering  
\includegraphics[width=1\linewidth]{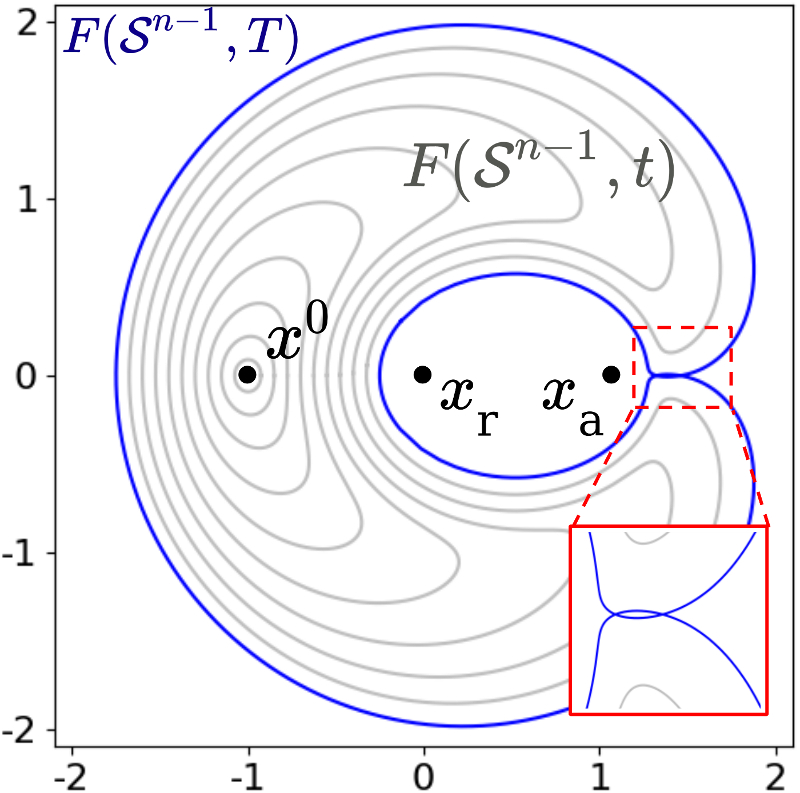}
\end{minipage}
\hspace{0.03\linewidth}
\begin{minipage}{0.32\linewidth}
\captionof{figure}{Solutions of \PMPODE{d^0} for all directions $d^0\in\cS^{n-1}$ at different times $t\in[0,T]$ for the attraction-repulsion system in Example \ref{example:selfintersect}.}
\label{fig:selfintersections}
\end{minipage}

Consider the $2$d dynamics $\dot{x}(t)=f(x(t))+w(t)$: $f(x)=\frac{x_\textrm{a}-x}{\|x_\textrm{a}-x\|^3}-\frac{x_\textrm{r}-x}{\|x_\textrm{r}-x\|^3}$ gives attraction-repulsion 
\ifarxiv
    terms\footnote{To satisfy Assumption \ref{assumption:f}, $f(\cdot)$ can be composed with a smooth cut-off function so it is well-defined at $x_\textrm{a}$ and $x_\textrm{r}$.}, 
\fi
$x_\textrm{a},x_\textrm{r},x^0\in\R^2$, %
and $\W=B(0,0.1)$. %
The sets $F(\cS^{n-1},t)$ are shown in Figure \ref{fig:selfintersections}. 
For $t$ large-enough, $F(\cS^{n-1},t)$ is not a submanifold of $\R^n$ since it self-intersects. Yet, Assumptions \ref{assumption:f}-\ref{assumption:X0} (and thus Theorems \ref{thm:hull_F} and \ref{thm:error_bound}) hold. 
Intuitively, when taking the convex hulls of $F(\cS^{n-1},t)$, intersections vanish. %
Practically, this convexity leads to a characterization that holds for arbitrarily-large times $T$, compared to results in \cite{Krener1989} that study the structure of the true reachable sets and rely on small-time assumptions.
\end{tcolorbox}
\end{figure}

\section{Proofs of Theorem \ref{thm:error_bound} and Lemma \ref{lem:partial_hull_reachable_submanifold}}
\label{sec:error_bound:proof}
First, we prove that $\partial\hull(\X_t)$ is a submanifold of $\R^n$ of dimension $(n-1)$ (Lemma  \ref{lem:partial_hull_reachable_submanifold}). The analysis leverages Theorem \ref{thm:hull_F} and properties of convex sets.

\textit{Proof of Lemma  \ref{lem:partial_hull_reachable_submanifold}:} 
$\hull(\X_t)$ is convex, compact (Lemma \ref{lem:Y_is_compact}), and has interior points ($\Int(\X_t)\neq\emptyset$ since the reachable state $x_w(t)$ associated to the disturbance $w(\cdot)=0$ is clearly in $\Int(\X_t)$). %
Thus, by \cite[Theorem 2.2.4]{Schneider2014}, it suffices to prove that there is a unique support hyperplane to $\hull(\X_t)$ at any boundary point $x\in\partial\hull(\X_t)$. %
In the following, as in the proof of Lemma \ref{lem:partial_HY_subset_F_subset_Y}, we prove the result for $t=T$  without loss of generality. 

First, let $x\in\partial\hull(\X_T)\cap F(\cS^{n-1},T)$. 
As $x\in\partial\hull(\X_T)$, 
by Lemma \ref{lem:hyperplane}, there exists a support hyperplane $\{y\in\R^n:d^\top (y-x)=0\}$ for $\hull(\X_T)$ at $x$ parameterized by some 
$d\in\cS^{n-1}$ such that $d^\top x\geq d^\top y$ for all $y\in\hull(\X_T)$. In particular, since $\X_T\subseteq\hull(\X_T)$, 
\begin{equation}\label{eq:normal_inequality}
d^\top x\geq d^\top y\quad\text{for all}\quad y\in\X_T.
\end{equation} 
As $x\in F(\cS^{n-1},T)$, there exists some $d^0\in\cS^{n-1}$ and $(x_w,p_w,w)$ that solve \PMPODE{d^0} with  $x=x_w(T)$. 
Then, by \eqref{eq:normal_inequality}, $d^\top x=d^\top x_w(T)\geq d^\top y$ for all $y\in\X_T$, so $(x_w,w)$ solves \OCP{d}. Define $q_w(t)=p_w(t)/\|p_w(t)\|$ as in \eqref{eq:q}. Then, $q_w(T)=d$ by \eqref{eq:pmp:pT}, so $w(T)
=(n^{\partial\W})^{-1}(g(T,x)^\top d/\|g(T,x)^\top d\|)$.

\begin{wrapfigure}{R}{0.41\linewidth}
\begin{minipage}{1\linewidth}
    \centering  
\includegraphics[width=1\linewidth]{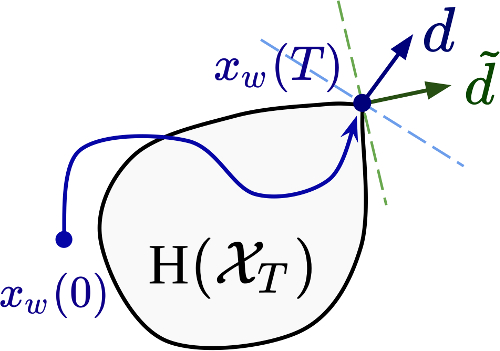}
\caption{Two support hyperplanes at $x{=}x_w(T)$.}
\label{fig:two_support_hyperplanes}
\end{minipage}
\end{wrapfigure}

By contradiction (see Figure \ref{fig:two_support_hyperplanes}), assume that there is a different support hyperplane for $\hull(\X_T)$ at $x=x_w(T)$  parameterized by $\tilde{d}\in\cS^{n-1}$ with $\tilde{d}\neq d$. Then, since $\tilde{d}^\top x\geq \tilde{d}^\top y$ for all $y\in\X_T$, from the previous reasoning, the same trajectory $(x_w,w)$ also solves \OCP{\tilde{d}}. 
Thus, $w$ satisfies $w(T)=(n^{\partial\W})^{-1}(g(T,x)^\top\tilde{d}/\|g(T,x)^\top\tilde{d}\|)$. 
This is a contradiction, since $g(T,x)$ is invertible by Assumption \ref{assumption:g}, $n^{\partial\W}$ is a diffeomorphism by Assumption \ref{assumption:W}, and $\tilde{d}\neq d$. %
Thus, $x$ has a unique support hyperplane.

Second, we consider boundary points that are not in $F(\cS^{n-1},T)$. Let $x\in\partial\hull(\X_T)\setminus F(\cS^{n-1},T)$. 
As $\text{Ext}(\hull(F(\cS^{n-1},T)))\subseteq F(\cS^{n-1},T)$ (Lemma \ref{lem:extreme_points_hull}) and $\hull(\X_T)=\hull(F(\cS^{n-1},T))$ (Theorem \ref{thm:hull_F}), $x$ is not an extreme point. Thus, $x$ can be written as $x=\alpha x_1+(1-\alpha)x_2$ for some $\alpha\in(0,1)$,  $x_1\in\hull(\X_T)$, and $x_2\in\partial\hull(\X_T)\cap F(\cS^{n-1},T)$ (see the proof of Lemma \ref{lem:extreme_points_hull} and note that $\hull(\X_T)=\hull(\partial\hull(\X_T)\cap F(\cS^{n-1},T))$ by Theorem \ref{thm:hull_F} and Lemma \ref{lem:extreme_points_hull}). Since $x\in\partial\hull(\X_T)$, there is a support hyperplane $H=\{y\in\R^n:d^\top (y-x)=0\}$  at $x$ parameterized by some $d\in\cS^{n-1}$ such that $d^\top x\geq d^\top y$ for all $y\in\hull(\X_T)$. Thus,
$$
x=\alpha x_1+(1-\alpha)x_2, \ \ 
d^\top x\geq d^\top x_1, \ \ 
d^\top x\geq d^\top x_2,
$$
from which one can show that $d^\top x=d^\top x_1=d^\top x_2$. Thus, $H$ is a support hyperplane at $x_2$. Since $x_2$ has a unique support hyperplane as shown previously, we conclude that $H$ is the unique support hyperplane at $x$. This concludes the proof of Lemma  \ref{lem:partial_hull_reachable_submanifold}.
\hfill $\blacksquare$ %

The proof of Theorem \ref{thm:error_bound} relies on the fact that $F$ maps tangent spaces of $\cS^{n-1}$ to tangent spaces of $\partial\hull(\X_t)$.
\begin{lemma}[Tangent vectors map to tangent vectors]\label{lem:tangent_to_tangent} 
Assume that $f$, $g$, $\W$, and $\X_0$ satisfy Assumptions \ref{assumption:f}-\ref{assumption:X0}. 
Let $t\in(0,T]$, and define $F^t(\cdot)=F(\cdot,t)$. Then, 
\begin{equation}
\dd F^t_{d^0}(T_{d^0}\cS^{n-1})\subseteq T_x\partial\hull(\X_t)
\end{equation}
for all 
$x\in\partial\hull(\X_t)\cap F(\cS^{n-1},t)$, where $x=F(d^0,t)$ and $d^0\in\cS^{n-1}$.
\end{lemma}
\begin{proof}
By Lemma \ref{lem:partial_hull_reachable_submanifold}, each tangent space $T_x\partial\hull(\X_t)$ is well-defined and of dimension $(n-1)$. Next, let $t\in(0,T]$ and $x\in\partial\hull(\X_t)\cap F(\cS^{n-1},t)$ be such that $x=F(d^0,t)$ for some $d^0\in\cS^{n-1}$. 

By contradiction, let $v\in T_{d^0}\cS^{n-1}$ be a tangent vector such that $\dd F^t_{d^0}(v)\notin T_x\partial\hull(\X_t)$. Then, there exists a smooth curve $\gamma:(-\epsilon,\epsilon)\to\cS^{n-1}$ such that $\gamma(0)=d^0$ and $\gamma'(0)=v$. Define the smooth curve $\alpha:(-\epsilon,\epsilon)\to\R^n$ by $\alpha(r)=F^t(\gamma(r))$ and note that $\alpha(r)\in\X_t$ for all $r\in(-\epsilon,\epsilon)$ by Theorem \ref{thm:hull_F}. 
Since $\alpha'(0)=\dd F^t_{d^0}(v)\notin T_x\partial\hull(\X_t)$, by \cite[Lemma 4.6]{LewBonalliJansonPavone2023}, there exists some $s\in(-\epsilon,\epsilon)$ such that $\alpha(s)\notin\hull(\X_t)$. This is a contradiction, since $\alpha(s)\in\X_t\subseteq\hull(\X_t)$. 
\end{proof}

Finally, Theorem \ref{thm:error_bound} follows from combining Lemmas  \ref{lem:partial_hull_reachable_submanifold}-\ref{lem:tangent_to_tangent} and recent geometric results in \cite{LewBonalliJansonPavone2023}.

\textit{Proof of Theorem \ref{thm:error_bound}:}
First, each tangent space $T_x\partial\hull(\X_t)$ at $x\in\partial\hull(\X_t)$ is well-defined by Lemma \ref{lem:partial_hull_reachable_submanifold}. 
Then, for all $x\in\partial\hull(\X_t)\cap F(\cS^{n-1},t)$ and $d^0,z\in\cS^{n-1}$ with $x=F(d^0,t)$,
\begin{equation}\label{thm:error_bound:tangent_bound}
d_{T_x\partial\hull(\X_t)}(F(z,t)-x)\leq\frac{1}{2}(\bar{L}_t+\bar{H}_t)\|z-d^0\|^2
\end{equation}
by \cite[Lemma 4.4]{LewBonalliJansonPavone2023} %
and due to Lemma \ref{lem:tangent_to_tangent} (the proof of \cite[Lemma 4.4]{LewBonalliJansonPavone2023} applies to our setting by replacing the use of\cite[Lemma 4.7]{LewBonalliJansonPavone2023} with Lemma \ref{lem:tangent_to_tangent}). Moreover, applying \cite[Lemma 4.3]{LewBonalliJansonPavone2023} gives
\begin{align}
\label{thm:error_bound:hausdorff_tangent}
&d_H(\hull(\X_t),\hull(F(Z_\delta,t)))\hfill
\\
&
\hfill\leq
\sup_{x\in\partial\hull(\X_t)\cap F(\cS^{n-1},t)}
\left(
\inf_{z\in Z_\delta}d_{T_x\partial\hull(\X_t)}(F(z,t)-x)
\right).
\nonumber
\end{align}
The conclusion follows from \eqref{thm:error_bound:tangent_bound} and \eqref{thm:error_bound:hausdorff_tangent}.
\hfill $\blacksquare$ %
\section{Approximate characterization for rectangular uncertainty sets $\W$ and $\X_0$}\label{sec:rectangular}
In the next three sections, we relax Assumptions \ref{assumption:g}-\ref{assumption:X0}. 
First, we relax Assumptions \ref{assumption:W} and \href{assumption:X0}{(A4b)}, which state that $\partial\W$ and $\partial\X_0$ are ovaloids. These assumptions prevent using rectangular uncertainties, as defined below. %
\begin{tcolorbox}[
title={\textit{Assumption \customlabel{assumption:W_box}{5}($\W$ is a hyper-rectangle)}
\phantom{$a^1$}
},
coltitle=black,
boxrule=2pt,
colframe=blue!6, %
enhanced, colback=blue!2, boxsep=0pt, left=6pt, right=6pt]
Let $\delta\bar{w}\in\R^n$. The set of disturbances is given as 
$$
\W=\{
w\in\R^n:
  |w_i|\leq\delta\bar{w}_i\ \forall i=1,\dots,n
\}.
$$
\end{tcolorbox}
\begin{tcolorbox}[
title={\textit{Assumption \customlabel{assumption:X0_box}{6}($\X_0$ is a hyper-rectangle)}
\phantom{$a^1$}
},
coltitle=black,
boxrule=2pt,
colframe=blue!6, %
enhanced, colback=blue!2, boxsep=0pt, left=6pt, right=6pt]
Let $\bar{x}_0,\delta\bar{x}_0\in\R^n$. The set of initial states is 
$$
\X_0=\{
x\in\R^n:
  |x_i-\bar{x}_{0i}|\leq\delta\bar{x}_{0i}
\ \forall i=1,\dots,n\}.
$$
\end{tcolorbox}
We propose an approximation scheme for problems with hyper-rectangular sets $(\W,\X_0)$ satisfying Assumptions \ref{assumption:W_box} and \ref{assumption:X0_box}. Given a relaxation parameter $\lambda>1$, we use smooth inner- and outer-approximations 
of $(\W,\X_0)$, \rev{defined in \eqref{eq:Wlambda_X0lambda_approximations}  as $\lambda$-norm ellipsoids}. 
The approximation scheme is shown in Figure  \ref{fig:rectangle_smooth_approx} and has three key properties. 
\begin{itemize}
\item
The approximations \rev{of $(\W,\X_0)$} 
satisfy Assumptions \ref{assumption:W} and \ref{assumption:X0}, so the convex hulls of the\rev{ir associated} reachable sets 
are characterized by Theorem \ref{thm:hull_F}. 
\item
The 
\rev{approximations} 
either inner- or outer-\rev{bound} $(\W,\X_0)$, 
so their reachable sets 
either inner- and outer-approximate the true reachable sets $\X_t$. 
\item
By choosing $\lambda$ large-enough, the \rev{approximations} 
can be made arbitrarily close to $(\W,\X_0)$, so the \rev{resulting} approximate reachable sets 
can be made arbitrarily close to the true reachable sets $\X_t$.
\end{itemize}
Combining these properties, we obtain arbitrarily-close inner- and outer-approximations of the convex hulls of the reachable sets $\X_t$ of dynamical systems with $(\W,\X_0)$ satisfying Assumptions \ref{assumption:W_box} and \ref{assumption:X0_box}. This characterization is given in Theorem \ref{thm:hull_F:rect_lambda} and is proved in Section \ref{sec:rectangular_proof}. Below, we state this result and necessary definitions.

Given a relaxation parameter $\lambda>1$, we define the maps 
$\left(\,\uwidehat{n^{\partial\W}_\lambda}\,\right)^{-1},
\left(\,\uwidehat{n^{\partial\X_0}_\lambda}\,\right)^{-1}:
\cS^{n-1}\to\R^n$ 
by 
\begin{subequations}
\begin{align}
\label{eq:gauss_map:Wlambda_under}
\left(\uwidehat{n^{\partial\W}_\lambda}\right)^{-1}(d)
&=
\frac{d\odot|d|^{\frac{2-\lambda}{\lambda-1}}\odot\delta\bar{w}^{\frac{\lambda}{\lambda-1}}}
{\big\||d\odot\delta\bar{w}|^{\frac{1}{\lambda-1}}\big\|_\lambda},
\\
\label{eq:gauss_map:X0lambda_under}
\left(\uwidehat{n^{\partial\X_0}_\lambda}\right)^{-1}(d)
&=
\bar{x}_0+
\frac{d\odot|d|^{\frac{2-\lambda}{\lambda-1}}\odot\delta\bar{x}_0^{\frac{\lambda}{\lambda-1}}}
{\big\||d\odot\delta\bar{x}_0|^{\frac{1}{\lambda-1}}\big\|_\lambda},
\end{align}
\end{subequations}
for any $d\in\cS^{n-1}$, the under-approximation ODE
\begin{tcolorbox}[
coltitle=black,
boxrule=2pt,
colframe=blue!6, %
enhanced, colback=blue!2, boxsep=1pt, left=6pt, right=6pt]
\vspace{-3mm}
\begin{align*}
\hspace{-2mm}
\customlabel{PMPODElambda_under}{}\uwidehat{\textbf{ODE}^\lambda_{d^0}}{:} \ \   
\begin{aligned}
\dot{x}(t)&=\eqref{eq:ODE},
\quad 
\dot{p}(t)=\eqref{eq:pmpode:p_dot},
\quad t\in[0,T],
\\
w(t)&=
\left(\,\uwidehat{n^{\partial\W}_\lambda}\,\right)^{-1}\left(\frac{g(t,x(t))^\top p(t)}{\|g(t,x(t))^\top p(t)\|}\right),
\\
x(0)&=
\left(\,\uwidehat{n^{\partial\X_0}_\lambda}\,\right)^{-1}(d^0), \quad 
p(0)=d^0,
\end{aligned}
\end{align*}
\end{tcolorbox}
\noindent 
the maps 
$\left(\,\widehat{n^{\partial\W}_\lambda}\,\right)^{-1},
\left(\,\widehat{n^{\partial\X_0}_\lambda}\,\right)^{-1}:
\cS^{n-1}\to\R^n$ by
\begin{subequations}
\begin{align}
\label{eq:gauss_map:Wlambda_over}
\left(\widehat{n^{\partial\W}_\lambda}\right)^{-1}(d)
&=
n^{\frac{1}{\lambda}}\frac{d\odot|d|^{\frac{2-\lambda}{\lambda-1}}\odot\delta\bar{w}^{\frac{\lambda}{\lambda-1}}}
{\big\||d\odot\delta\bar{w}|^{\frac{1}{\lambda-1}}\big\|_\lambda},
\\
\label{eq:gauss_map:X0lambda_over}
\left(\widehat{n^{\partial\X_0}_\lambda}\right)^{-1}(d)
&=
\bar{x}_0+
n^{\frac{1}{\lambda}}\frac{d\odot|d|^{\frac{2-\lambda}{\lambda-1}}\odot\delta\bar{x}_0^{\frac{\lambda}{\lambda-1}}}
{\big\||d\odot\delta\bar{x}_0|^{\frac{1}{\lambda-1}}\big\|_\lambda},
\end{align}
\end{subequations}
for any $d\in\cS^{n-1}$, and the over-approximation ODE
\begin{tcolorbox}[
coltitle=black,
boxrule=2pt,
colframe=blue!6, %
enhanced, colback=blue!2, boxsep=1pt, left=6pt, right=6pt]
\vspace{-3mm}
\begin{align*}
\hspace{-2mm}
\customlabel{PMPODElambda_over}{\widehat{\textbf{ODE}^\lambda_{d^0}}}
{:} \ \ 
\begin{aligned}
\dot{x}(t)&=\eqref{eq:ODE},
\quad 
\dot{p}(t)=\eqref{eq:pmpode:p_dot},
\quad t\in[0,T],
\\
w(t)&=
\left(\,\widehat{n^{\partial\W}_\lambda}\,\right)^{-1}\left(\frac{g(t,x(t))^\top p(t)}{\|g(t,x(t))^\top p(t)\|}\right),
\\
x(0)&=
\left(\,\widehat{n^{\partial\X_0}_\lambda}\,\right)^{-1}(d^0), \quad 
p(0)=d^0.
\end{aligned}
\end{align*}
\end{tcolorbox}

The result below approximately characterizes the convex hulls of reachable sets of systems with rectangular sets $\W$ and $\X_0$. Importantly,  the proposed approximations always inner- and outer-bound the true convex hulls and converge as the relaxation parameter $\lambda$ increases.

\begin{tcolorbox}[
title={\textit{Theorem  \customlabel{thm:hull_F:rect_lambda}{3}(Approximate characterization for hyper-rectangular uncertainty sets $\W$ and $\X_0$)}},
colbacktitle=blue!10, 
coltitle=black,
boxrule=3pt,
colframe=blue!10, %
breakable,
enhanced, colback=blue!4, boxsep=0pt, left=6pt, right=6pt]
Assume that $f$ and $g$ satisfy Assumptions \ref{assumption:f} and \ref{assumption:g}, and that $\W$ and $\X_0$ satisfy Assumptions \ref{assumption:W_box} and \ref{assumption:X0_box}. 
Given any $\lambda>1$ and direction $d^0\in\cS^{n-1}$, we define the augmented ODEs \PMPODElambdaunder{d^0} and \PMPODElambdaover{d^0}, with unique solutions
$(\uwidehat{x}_{d^0}^\lambda,\uwidehat{p}_{d^0}^\lambda)$ and $(\widehat{x}_{d^0}^\lambda,\widehat{p}_{d^0}^\lambda)$, respectively.
Define the two maps
\begin{align}
\label{eq:Flambdas}
\hspace{-1mm}
\uwidehat{F_\lambda}:
\cS^{n-1}\times[0,T]\to\R^n: 
(d^0,t)\mapsto\uwidehat{x}_{d^0}^\lambda(t),
\\
\hspace{-1mm}
\widehat{F_\lambda}:
\cS^{n-1}\times[0,T]\to\R^n: 
(d^0,t)\mapsto\widehat{x}_{d^0}^\lambda(t). 
\end{align}
Then, for all $t\in[0,T]$,
\begin{equation}\label{eq:hull_F:rect_lambda:subsets}
\hull\left(\uwidehat{F_\lambda}(\cS^{n-1},t)\right)
\subseteq
\hull(\X_t)
\subseteq
\hull\left(\widehat{F_\lambda}(\cS^{n-1},t)\right)
\end{equation}
and as $\lambda\to\infty$,
\begin{subequations}
\begin{align}
\label{eq:hull_F_lambda_converges_under}
d_H\left(
\hull(\X_t),
\hull\left(\uwidehat{F_\lambda}(\cS^{n-1},t)\right)
\right)\to 0,
\\[1mm]
\label{eq:hull_F_lambda_converges_over}
d_H\left(
\hull(\X_t),
\hull\left(\widehat{F_\lambda}(\cS^{n-1},t)\right)
\right)\to 0.
\end{align}
\end{subequations}
\end{tcolorbox}

\section{Proof of Theorem \ref{thm:hull_F:rect_lambda}}
\label{sec:rectangular_proof}
First, we describe the smooth set approximation %
used in Theorem \ref{thm:hull_F:rect_lambda}. 
Given $\bar{x}\in\R^n$ and $\delta\bar{x}\in\R^n$ with $\delta\bar{x}_i>0$ for all $i=1,\dots,n$, we define the hyper-rectangular set
\begin{align}
C&=
\left\{x\in\R^n: |x_i-\bar{x}_i|\leq \delta\bar{x}_i \text{ for all }i=1,\dots,n
\right\}
\nonumber
\\
&=
\left\{x\in\R^n: h(x)\leq 1\right\}
\label{eq:Crectangular}
\end{align}
with the continuous function $h:\R^n\to\R$ defined as
\begin{equation}\label{eq:C:h}
h(x)=
\|(x-\bar{x})\odot\delta\bar{x}^{-1}\|_\infty^2
=
\max_{i=1,\dots,n}\left(|x_i-\bar{x}_i|/\delta\bar{x}_i\right)^2.
\end{equation}
For any $\lambda>1$, we define the function $h_\lambda:\R^n\to\R$ as
\begin{equation}\label{eq:C:hlambda}
h_\lambda(x)=
\|(x-\bar{x})\odot\delta\bar{x}^{-1}\|_\lambda^2
=
\left(
\sum_{i=1}^n
\left|\frac{x_i-\bar{x}_i}{\delta\bar{x}_n}\right|^\lambda\right)^\frac{2}{\lambda},
\end{equation}
and, as shown in Figure \ref{fig:rectangle_smooth_approx}, the associated sets 
\begin{subequations}
\begin{align}
\uwidehat{C_\lambda} &=
\{x\in\R^n:h_\lambda(x)\leq 1\},
\label{eq:Clambda_under}
\\
\widehat{C_\lambda} &=
\{x\in\R^n:h_\lambda(x)\leq n^{\frac{2}{\lambda}}\}.
\label{eq:Clambda_over}
\end{align}
\end{subequations}
We define the map $n^{\partial C_\lambda}:(\R^n\setminus\{0\})\to\cS^{n-1}$ by 
\begin{equation}\label{eq:gauss_map:Clambda}
n^{\partial C_\lambda}(x)
=
\frac{(x-\bar{x})\odot|x-\bar{x}|^{\lambda-2}\odot\delta\bar{x}^{-\lambda}}{\||x-\bar{x}|^{\lambda-1}\odot\delta\bar{x}^{-\lambda}\|},
\end{equation}
and the maps 
$(n^{\partial\uwidehat{C_\lambda}})^{-1},
(n^{\partial\widehat{C_\lambda}})^{-1}:
\cS^{n-1}\to\R^n$ 
by 
\begin{subequations}
\begin{align}
\label{eq:gauss_map:Clambda_under:inverse}
\Big(n^{\partial\uwidehat{C_\lambda}}\Big)^{-1}(d)&=
\bar{x}+
\frac{d\odot|d|^{\frac{2-\lambda}{\lambda-1}}\odot\delta\bar{x}^{\frac{\lambda}{\lambda-1}}}
{\big\||d\odot\delta\bar{x}|^{\frac{1}{\lambda-1}}\big\|_\lambda},
\\
\label{eq:gauss_map:Clambda_over:inverse}
\left(n^{\partial\widehat{C_\lambda}}\right)^{-1}(d)&=
\bar{x}+
n^{\frac{1}{\lambda}}\frac{d\odot|d|^{\frac{2-\lambda}{\lambda-1}}\odot\delta\bar{x}^{\frac{\lambda}{\lambda-1}}}
{\big\||d\odot\delta\bar{x}|^{\frac{1}{\lambda-1}}\big\|_\lambda}.
\end{align}
\end{subequations}

\begin{figure}[!t]
\centering
\includegraphics[width=0.95\linewidth]{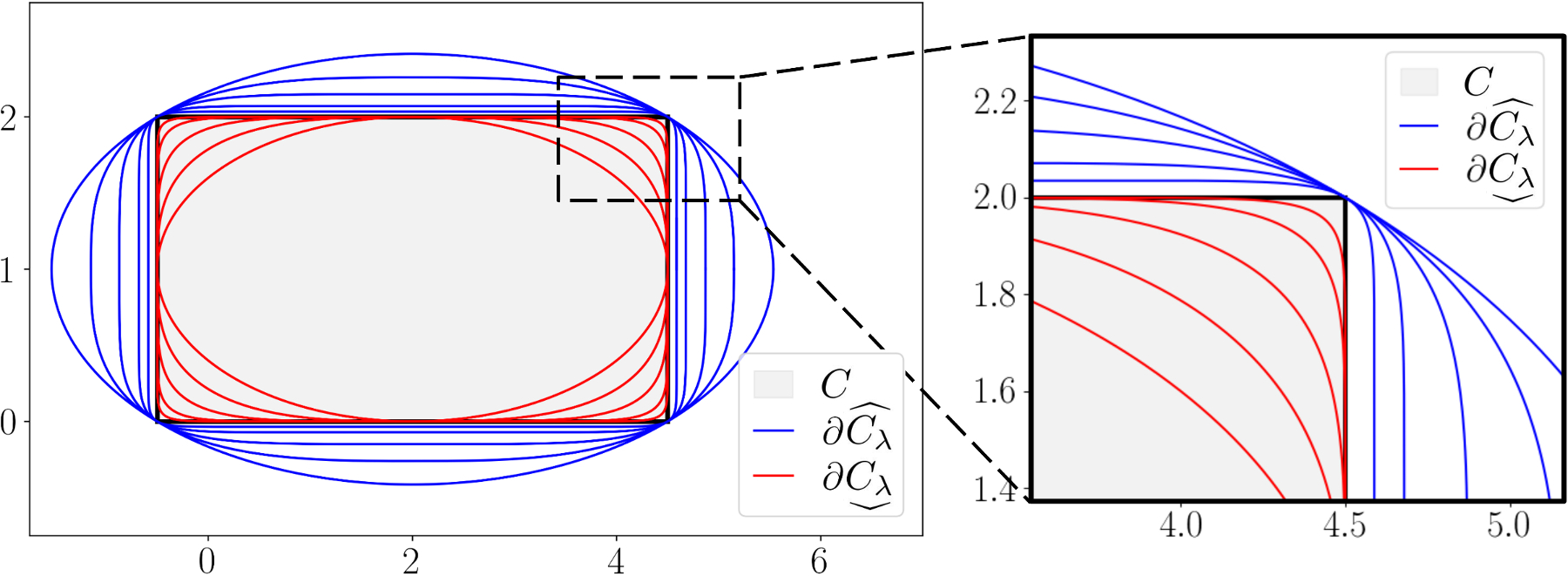}
\caption{Smooth under- and over-approximations of a rectangular set $C$ for different relaxation parameters $\lambda$. As $\lambda$ increases, the approximations converge to the set $C$.}
\label{fig:rectangle_smooth_approx}
\end{figure}

\begin{lemma}[The sets $(\uwidehat{C_\lambda},\widehat{C_\lambda})$ approximate $C$] 
\label{lem:Clambda}
Let $\lambda>1$ and define the sets 
$(C,\uwidehat{C_\lambda},\widehat{C_\lambda})$ %
as in \eqref{eq:Crectangular}-\eqref{eq:Clambda_over}.
\begin{itemize}
\item
The sets $(C,\uwidehat{C_\lambda},\widehat{C_\lambda})$ are convex, compact and
\begin{equation}
\uwidehat{C_\lambda}
\subseteq
C\subseteq
\widehat{C_\lambda}.
\end{equation}

\item
The sets $(\uwidehat{C_\lambda},\widehat{C_\lambda})$ approximate $C$ arbitrarily well by increasing $\lambda$:
\begin{subequations}
\begin{align}
d_H(\uwidehat{C_\lambda},C)\to 0\text{ as }\lambda\to\infty,
\\
d_H(\widehat{C_\lambda},C)\to 0\text{ as }\lambda\to\infty.
\end{align}

\item
The boundaries $(\partial\uwidehat{C_\lambda},\partial\widehat{C_\lambda})$ are ovaloids. 
Their Gauss maps are given by $n^{\partial\uwidehat{C_\lambda}}(x)=n^{\partial C_\lambda}(x)$ and $n^{\partial\widehat{C_\lambda}}(x)=n^{\partial C_\lambda}(x)$, where the map $n^{\partial C_\lambda}$ is defined in \eqref{eq:gauss_map:Clambda}. 
Their inverse Gauss maps $(n^{\partial\uwidehat{C_\lambda}})^{-1}$ and $(n^{\partial\widehat{C_\lambda}})^{-1}$ are given by
\eqref{eq:gauss_map:Clambda_under:inverse} and \eqref{eq:gauss_map:Clambda_over:inverse}.
\end{subequations}
\end{itemize}
\end{lemma}
The proof of Lemma \ref{lem:Clambda} is provided in Appendix \ref{sec:lem:Clambda:proof}. With this result, we prove Theorem \ref{thm:hull_F:rect_lambda}.

\textit{Proof of Theorem \ref{thm:hull_F:rect_lambda}:} 
First, we define the sets
\begin{subequations}\label{eq:Wlambda_X0lambda_approximations}
\begin{align}
\uwidehat{\W^\lambda}&=\{w\in\R^n:h_\lambda^\W(w)\leq 1\},
\\
\uwidehat{\X_0^\lambda}&=\{x\in\R^n:h_\lambda^\X(x)\leq 1\},
\\
\widehat{\W^\lambda}&=\{w\in\R^n:h_\lambda^\W(w)\leq n^{\frac{2}{\lambda}}\},
\\
\widehat{\X_0^\lambda}&=\{x\in\R^n:h_\lambda^\X(x)\leq n^{\frac{2}{\lambda}}\},
\end{align}
\end{subequations}
where $h_\lambda^\W(w)=
\|w \odot \delta\bar{w}^{-1}\|_\lambda^2$ and $h_\lambda^\X(x)=
\|(x-\bar{x}_0)\odot\delta\bar{x}^{-1}_0\|_\lambda^2$ as in \eqref{eq:C:hlambda}. 
By Lemma \ref{lem:Clambda}, 
$(\uwidehat{\W^\lambda},\widehat{\W^\lambda})$ satisfy Assumption \ref{assumption:W}, 
$(\uwidehat{\X_0^\lambda},\widehat{\X_0^\lambda})$ satisfy Assumption \ref{assumption:X0},
\begin{subequations}
\begin{alignat}{2}
\label{eq:Wlambda_X0lambda_under_over_approx}
&\quad\, 
\uwidehat{\W^\lambda}
\subseteq\W\subseteq
\widehat{\W^\lambda},
\quad
\uwidehat{\X_0^\lambda}
\subseteq\X_0\subseteq
\widehat{\X_0^\lambda},
\\
\label{eq:Wlambda_converges_to_W}
&\lim_{\lambda\to 0}
d_H(\uwidehat{\W^\lambda},\W)
=
\lim_{\lambda\to 0}
d_H(\widehat{\W^\lambda},\W)\,
=0,
\\
\label{eq:X0lambda_converges_to_X0}
&\lim_{\lambda\to 0}
d_H(\uwidehat{\X_0^\lambda},\X_0)
=
\lim_{\lambda\to 0}
d_H(\widehat{\X_0^\lambda},\X_0)
=0,
\end{alignat}
\end{subequations}
and the inverse Gauss maps of $(\uwidehat{\W^\lambda},\uwidehat{\X_0^\lambda},\widehat{\W^\lambda},\widehat{\X_0^\lambda})$ 
are given by 
(\eqref{eq:gauss_map:Wlambda_under},
\eqref{eq:gauss_map:X0lambda_under},
\eqref{eq:gauss_map:Wlambda_over},
\eqref{eq:gauss_map:X0lambda_over}). 

Second, we define the reachable sets
\begin{align*}
\uwidehat{\X_t^\lambda}=
\left\lbrace
x_{(w,x^0)}(t): 
w\in L^\infty([0,T],\uwidehat{\W^\lambda}),
\, 
x^0\in\uwidehat{\X_0^\lambda}
\right\rbrace,
\\[1mm]
\widehat{\X_t^\lambda}=
\left\lbrace
x_{(w,x^0)}(t): 
w\in L^\infty([0,T],\widehat{\W^\lambda}),
\, 
x^0\in\widehat{\X_0^\lambda}
\right\rbrace.
\end{align*}
By definition and thanks to \eqref{eq:Wlambda_X0lambda_under_over_approx}, for all $t\in[0,T]$,
\begin{equation}
\label{eq:Xlambda_under_over_approx}
\uwidehat{\X_t^\lambda}\subseteq\X_t\subseteq\widehat{\X_t^\lambda}.
\end{equation}
By Theorem \ref{thm:hull_F}, for all $t\in[0,T]$,
\begin{subequations}
\begin{align}
\label{eq:hull_Flambda_equals_hull_Xt_under}
\hull\left(\uwidehat{\X_t^\lambda}\right)
&=
\hull\left(\uwidehat{F_\lambda}(\cS^{n-1},t)\right),
\\
\label{eq:hull_Flambda_equals_hull_Xt_over}
\hull\left(\widehat{\X_t^\lambda}\right)
&=
\hull\left(\widehat{F_\lambda}(\cS^{n-1},t)\right).
\end{align}
\end{subequations}
Combining \eqref{eq:Xlambda_under_over_approx}, \eqref{eq:hull_Flambda_equals_hull_Xt_under}, and \eqref{eq:hull_Flambda_equals_hull_Xt_over}  gives \eqref{eq:hull_F:rect_lambda:subsets}.
Combining \eqref{eq:Wlambda_converges_to_W},  \eqref{eq:X0lambda_converges_to_X0} and a standard continuity result (see Lemma \ref{lem:stability_reachset_errors} in the appendix) %
gives  \eqref{eq:hull_F_lambda_converges_under} and \eqref{eq:hull_F_lambda_converges_over}.
\phantom{a} \hfill $\blacksquare$

\section{Approximate characterization for non-invertible $g(t,x)$}\label{sec:noninvertible}
Assumption \ref{assumption:g} states that $g(t,x)$ is invertible, so Theorem \ref{thm:hull_F} does not directly apply to problems that have more states than disturbances. To relax this assumption, given two integers $m<n$, we consider the system
\begin{align}
\label{eq:ODE:not_full_rank}
\dot{x}(t)&=f(t,x(t))+\sum_{i=1}^mg_i(t,x(t))w_i(t)
\nonumber
\\
&=f(t,x(t))+g(t,x(t))w(t),
\ t\in[0,T],
\end{align}
where 
$x(0)\in\X_0$ with $\X_0\subset\R^n$ satisfying Assumption \ref{assumption:X0}, 
$w\in L^\infty([0,T],\W)$ with $\W\subset\R^m$ satisfying Assumption \ref{assumption:W}, $f$ satisfies Assumption \ref{assumption:f}, and $g(t,x)=(g_1,\dots,g_m)(t,x)\in\R^{n\times m}$. The reachable sets of \eqref{eq:ODE:not_full_rank} are   defined as in \eqref{eq:Y_reachable_set} and are denoted by $\X_t$.  We relax the invertibility assumption on $g$ (Assumption \ref{assumption:g}) as follows.

\begin{tcolorbox}[
title={\textit{Assumption \customlabel{assumption:g:full_rank}{7}($g$ is full rank)}
\phantom{$a^1$}
},
coltitle=black,
boxrule=2pt,
colframe=blue!6, %
enhanced, colback=blue!2, boxsep=0pt, left=6pt, right=6pt]
$g(t,x)$ %
is full rank for all $(t,x)\in[0,T]\times\R^n$.
\end{tcolorbox}
Assumption \ref{assumption:g:full_rank} is standard and holds in many  practical applications. 
By appropriately completing the range of $g$, we approximate the system  \eqref{eq:ODE:not_full_rank} with a system that has similar reachable sets and satisfies Assumption \ref{assumption:f}-\ref{assumption:X0}. 
First, we rely on the choice of a particular set $\widehat{\W}\subset\R^n$. Below,   
$\pi(w_1,\dots,w_m,w_{m+1},\dots,w_n)=(w_1,\dots,w_m)$ denotes the projection map.

\begin{figure}[!t]
\centering
\includegraphics[width=0.8\linewidth]{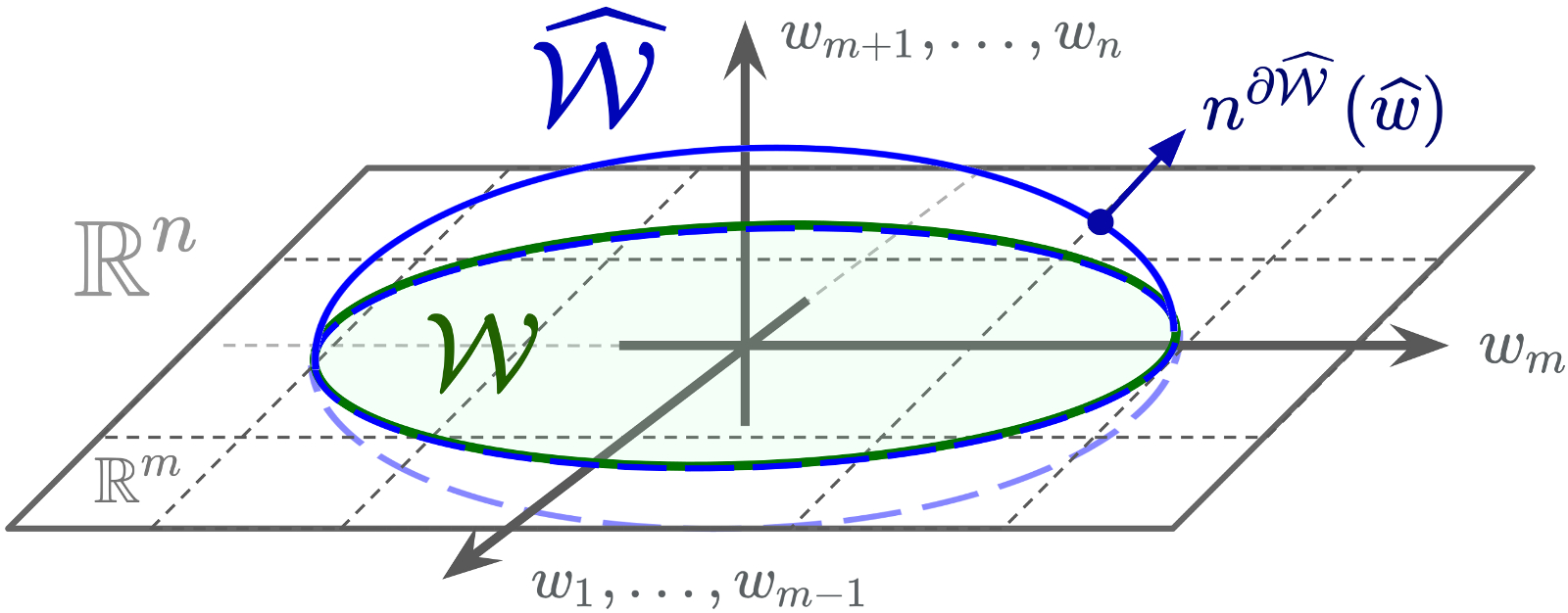}
\caption{Smooth approximation $\widehat{\W}$ satisfying Assumption \ref{assumption:Wbar:smooth_fulldim}.}
\label{fig:noninvertible:What}
\end{figure}

\begin{tcolorbox}[
title={\textit{Assumption \customlabel{assumption:Wbar:smooth_fulldim}{8}($\widehat{\W}$ smoothly approximates $\W$ in $\R^n$)}
},
coltitle=black,
boxrule=2pt,
colframe=blue!6, %
enhanced, colback=blue!2, boxsep=0pt, left=6pt, right=6pt]
The set $\widehat{\W}\subset\R^n$ satisfies Assumption \ref{assumption:W}. Moreover, 
$\W\times\{0\}\subset\widehat{\W}$ and $\pi(\widehat{\W})=\W$.  %
\end{tcolorbox}

\begin{tcolorbox}[
title={\textit{Example \customlabel{example:gauss_maps:not_full_rank}{3}(%
Set $\widehat{\W}$ satisfying 
Assumption \ref{assumption:Wbar:smooth_fulldim})}},
colbacktitle=gray!20, 
coltitle=black,
boxrule=3pt,
colframe=gray!20, %
breakable,
boxsep=0pt, before skip=0pt, after skip=0pt]
If $\W=B(0,1)\subset\R^m$, then $\widehat{\W}=B(0,1)\subset\R^n$ satisfies Assumption \ref{assumption:Wbar:smooth_fulldim}. More generally, if 
the set
$$
\W=\{w\in\R^m:h(w)\leq 1\}
$$
satisfies Assumption \ref{assumption:W}, then the set $\widehat{\W}$ defined as
$$
\widehat{\W}=\left\{w\in\R^n:h(w_{1:m})+\frac{1}{2}\|w_{m+1:n}\|^2\leq 1\right\}
$$
satisfies Assumption \ref{assumption:Wbar:smooth_fulldim}, where we use the notation $w=(w_{1:m},w_{m+1:n})\in\R^n$. If also $h(\nabla h^{-1}(n^{\partial\W}(w)))=1/\|\nabla h(w)\|^2$ for all $w\in\partial\W$ (see Example \ref{example:gauss_maps} and Appendix \ref{apdx:sec:gauss_map_level_set:inv} for details), then the inverse Gauss map of  $\partial\widehat{\W}$ is (see \eqref{eq:gauss_map_level_set:inv})
$$
\left(n^{\partial\widehat{\W}}\right)^{-1}(d)
=
\nabla\widehat{h}^{-1}
\Big(
d\, /\, \widehat{h}\big(\nabla\widehat{h}^{-1}(d)\big)^{\frac{1}{2}}
\Big),
$$
where $\widehat{h}(w)=h(w_{1:m})+\frac{1}{2}\|w_{m+1:n}\|^2$ and 
$$
\nabla\widehat{h}^{-1}(d)=\left(\nabla h^{-1}(d_{1:m}),d_{m+1:n}\right).
$$
\end{tcolorbox}

Second, we define $\widehat{g}_\epsilon:\R\times\R^n\to\R^{n\times n}$ by
\begin{equation}\label{eq:g_bar_epsilon}
\widehat{g}_\epsilon(t,x)
=
\begin{bmatrix}
g_1\ \, \dots\ \, g_m\ \  
\epsilon g_{m+1}\ \, \dots\ \, \epsilon g_n
\end{bmatrix}(t,x),
\end{equation}
where $\epsilon>0$ and the functions $g_i:\R\times\R^n\to\R^n$ with $i=m+1,\dots,n$ are chosen as follows. %
\begin{tcolorbox}[
title={\textit{Assumption \customlabel{assumption:g_bar_epsilon}{9}($\|g_i\|=1$ and $\widehat{g}_\epsilon(t,x)$ is invertible)}
},
coltitle=black,
boxrule=2pt,
colframe=blue!6, %
enhanced, colback=blue!2, boxsep=0pt, left=6pt, right=6pt]
The functions $g_i:\R\times\R^n\to\R^n$ satisfy    $\|g_i(t,x)\|=1$ for all $i=m+1,\dots,n$ and are such that $\widehat{g}_\epsilon(t,x)$ is invertible for all $(t,x)\in[0,T]\times\R^n$.
\end{tcolorbox}
For $g(t,x)$ that is constant and satisfies Assumption \ref{assumption:g:full_rank}, choosing the maps $(g_{m+1},\dots,g_n)$ with constant values sampled at random on the sphere $\cS^{n-1}$ (from a uniform distribution) suffices to satisfy Assumption \ref{assumption:g_bar_epsilon}. We refer to \cite{Silva2010} for insightful discussion related to this assumption and to Section \ref{sec:numerical:dubins} for an example.

Third, we define the extended system
\begin{equation}
\label{eq:ODE:not_full_rank:extended}
\dot{x}(t)=f(t,x(t))+\widehat{g}_\epsilon(t,x(t))\widehat{w}(t),
\ t\in[0,T],
\end{equation}
where $x(0)\in\X_0$ and $\widehat{w}\in L^\infty([0,T],\widehat{\W})$, with associated reachable sets denoted by $\widehat{\X}_t^\epsilon$. 
The system \eqref{eq:ODE:not_full_rank:extended} satisfies the assumptions of Theorem \ref{thm:hull_F}. This suggests defining the following augmented ODE:

\begin{tcolorbox}[
title={\center\customlabel{PMPODEepsilon_full_rank}{$\textbf{ODE}^\epsilon_{d^0}$}},
coltitle=black,
boxrule=2pt,
colframe=blue!6, %
enhanced, colback=blue!2, boxsep=1pt, left=6pt, right=6pt]
\vspace{-3mm}
\begin{align*}
\dot{x}(t)&=f(t,x(t))+\widehat{g}_\epsilon(t,x(t))\widehat{w}(t),
\quad t\in[0,T],
\\
\dot{p}(t)&=
-(\nabla f(t,x(t))+\nabla\widehat{g}_\epsilon(t,x(t))\widehat{w}(t))^\top p(t),
\\
\widehat{w}(t)&=
\left(n^{\partial\widehat{\W}}\right)^{-1}\left(\frac{\widehat{g}_\epsilon(t,x(t))^\top p(t)}{\|\widehat{g}_\epsilon(t,x(t))^\top p(t)\|}\right),
\\
x(0)&=
\left(n^{\partial\X_0}\right)^{-1}(d^0), \quad 
p(0)=d^0,
\end{align*}
\end{tcolorbox}
As stated below, solutions to \PMPODEepsilonfullrank{d^0} characterize the convex hulls of the reachable sets $\X_t$ of the system \eqref{eq:ODE:not_full_rank} arbitrarily well by selecting $\epsilon$ small-enough.

\begin{tcolorbox}[
title={\textit{Theorem  \customlabel{thm:hull_F:not_full_rank}{4}(Approximate characterization for non-invertible $g(t,x)$)}},
colbacktitle=blue!10, 
coltitle=black,
breakable,
boxrule=3pt,
colframe=blue!10, %
enhanced, colback=blue!4, boxsep=0pt, left=6pt, right=6pt]
Consider the dynamical system in 
\eqref{eq:ODE:not_full_rank} with $m<n$. Assume that $f$ and $g$ satisfy Assumptions \ref{assumption:f} and \ref{assumption:g:full_rank}, 
and that $\W$ and $\X_0$ satisfy Assumptions \ref{assumption:W} and \ref{assumption:X0}. 
Choose a set $\widehat{\W}\subset\R^n$ satisfying
Assumption \ref{assumption:Wbar:smooth_fulldim}, $\epsilon>0$, and $(n-m)$ functions $g_i:\R\times\R^n\to\R^n$ 
satisfying Assumption \ref{assumption:g_bar_epsilon}. 
Given any direction $d^0\in\cS^{n-1}$, define $(x_{d^0}^\epsilon,p_{d^0}^\epsilon)$ as the solution to \PMPODEepsilonfullrank{d^0}.
Define the map
\begin{align}
\label{eq:Fepsilon_full_rank}
\hspace{-1mm}
F_\epsilon:
\cS^{n-1}\times[0,T]\to\R^n: 
(d^0,t)\mapsto x_{d^0}^\epsilon(t). 
\end{align}
Then, for all $t\in[0,T]$,
\begin{equation}\label{eq:hull_F:not_full_rank:subsets}
\hull(\X_t)
\subseteq
\hull\left(F_\epsilon(\cS^{n-1},t)\right),
\end{equation}
and there exists a constant $C_T^{\widehat{\W}}\geq 0$ such that
\begin{align}
\label{eq:hull_F:not_full_rank:Feps_converges_under}
d_H\left(
\hull(\X_t),
\hull\left(F_\epsilon(\cS^{n-1},t)\right)
\right)\leq C_T^{\widehat{\W}}\epsilon.
\end{align}
In particular, $\lim\limits_{\epsilon\to 0}
d_H(
\hull(\X_t),
\hull(F_\epsilon(\cS^{n-1},t))
)=0$.
\end{tcolorbox}
\begin{remark}[Error \rev{bounds and stable integration}] \rev{Error} bounds for obtaining convex hull approximations of the reachable sets of system 
\eqref{eq:ODE:not_full_rank} with Algorithm \ref{alg:1} \rev{can be derived by
combining \eqref{eq:error_bound} and \eqref{eq:hull_F:not_full_rank:Feps_converges_under}}. 

\rev{The disturbances $\widehat{w}(t)$ may become discontinuous as $\epsilon\to 0$, see  \cite{Silva2010}. Thus, a stable integration scheme should be used to integrate \PMPODEepsilonfullrank{d^0} for small values of $\epsilon$, and similarly for \PMPODElambdaunder{d^0} and \PMPODElambdaover{d^0} for large values of $\lambda$.}
\end{remark}

\begin{proof}
First, we define the intermediate system
\begin{equation}
\label{eq:ODE:not_full_rank:intermediate}
\dot{x}(t)=f(t,x(t))+\widehat{g}(t,x(t))\widehat{w}(t),
\ t\in[0,T] 
\end{equation}
with reachable sets $\widehat{\X}_t$, where $x(0)\in\X_0$, $\widehat{w}\in L^\infty([0,T],\widehat{\W})$, and the map $\widehat{g}:\R\times\R^n\to\R^{n\times n}$ is defined by \rev{$\widehat{g}(t,x)
=
[g_1(t,x),\dots,g_m(t,x),
0,\dots,0]$.}
%

\rev{Since} for any $\widehat{w}(t)\in\widehat{\W}$,  
we have $
\widehat{g}(t,x(t))\widehat{w}(t)=\sum_{i=1}^mg_i(t,x(t))\widehat{w}_i(t)
$  
and $(\widehat{w}_1,\dots,\widehat{w}_m)(t)\in\W$, 
the reachable sets $\widehat{\X}_t$ of system  \eqref{eq:ODE:not_full_rank:intermediate} satisfy\footnote{Note that Theorem \ref{thm:hull_F} does not give information about $\widehat{\X}_t$, since $\widehat{g}(t,x)$ is not invertible.}
\begin{equation}
\label{eq:hull_F:not_full_rank:barXt_is_Xt}
\widehat{\X}_t=\X_t 
\text{ for all } t\in[0,T].
\end{equation}

Second, given any $w(t)=(w_1,\dots,w_m)(t)\in\W$, the extended disturbance $\widehat{w}(t)=(w_1,\dots,w_m,0,\dots,0)(t)$ satisfies $\widehat{w}(t)\in\widehat{\W}$ and $\widehat{g}(t,x)\widehat{w}(t)=\widehat{g}_\epsilon(t,x)\widehat{w}(t)$. 
Thus, the reachable sets $\widehat{\X}_t^\epsilon$ of the system \eqref{eq:ODE:not_full_rank:extended} satisfy
\begin{equation}\label{eq:hull_F:not_full_rank:barXt_subset_barXteps}
\widehat{\X}_t\subseteq\widehat{\X}_t^\epsilon  
\text{ for all } t\in[0,T].
\end{equation}
Applying Theorem \ref{thm:hull_F} to the system \eqref{eq:ODE:not_full_rank:extended} gives $\hull(\widehat{\X}_t^\epsilon)=
\hull\left(F_\epsilon(\cS^{n-1},t)\right)$. Combining this last result with \eqref{eq:hull_F:not_full_rank:barXt_is_Xt} and \eqref{eq:hull_F:not_full_rank:barXt_subset_barXteps}
gives \eqref{eq:hull_F:not_full_rank:subsets}.

To show \eqref{eq:hull_F:not_full_rank:Feps_converges_under}, \rev{note} $
\|\widehat{g}-\widehat{g}_\epsilon\|_\infty\leq\epsilon
$, so $d_H(\widehat{\X}_t,\widehat{\X}_t^\epsilon)\leq C_T^{\widehat{\W}}\epsilon$ by a standard continuity result (Lemma \ref{lem:stability_reachset_errors} in the appendix). %
Thus, $d_H(\hull(\widehat{\X}_t),\hull(\widehat{\X}_t^\epsilon))\leq C_T^{\widehat{\W}}\epsilon$, so 
\rev{\eqref{eq:hull_F:not_full_rank:Feps_converges_under}} follows from  \eqref{eq:hull_F:not_full_rank:barXt_is_Xt} and $\hull(\widehat{\X}_t^\epsilon)=
\hull\left(F_\epsilon(\cS^{n-1},t)\right)$.
\end{proof}

\section{Results and applications}\label{sec:numerical}
We evaluate Algorithm \ref{alg:1} on three nonlinear systems, and use its reachable set estimates to design a robust model predictive controller (Algorithm \ref{alg:mpc}). 
Computation times are measured on a laptop with an 1.10GHz Intel Core i7-10710U CPU. Code to reproduce results is available at 
\scalebox{0.95}{\url{https://github.com/StanfordASL/chreach}}.

\subsection{Validating the relaxation schemes on Dubins car}\label{sec:numerical:dubins}
Consider the dynamical system with state $x(t)=(p_1,p_2,\theta)(t)\in\R^3$ evolving according to the ODE \rev{$\dot{x}(t)=(v\cos(\theta(t)),v\sin\theta(t)),\omega)+Gw(t)$ with $t\in[0,6]$,}   
$v=\omega=0.5$,  $G\in\R^{3\times m}$ a (constant) matrix, and $x(0)\in\X_0=\mathcal{E}(0,10^{-3}\cdot\text{diag}([1,1,10^{-1}]))$. 
We consider two different choices of $g$ and $\W$ that violate Assumption \ref{assumption:W} ($\partial\W$ is an ovaloid in $\R^n$) and Assumption \ref{assumption:g} ($g(t,x)$ is invertible), which allows us to validate the relaxation schemes in Section \ref{sec:rectangular} and \ref{sec:noninvertible}.

\subsubsection{Rectangular disturbances set}
Let $G=I_3$ and $\W\subset\R^3$ be the rectangular set in Assumption \ref{assumption:W_box} with $\delta\bar{w}=10^{-2}(1,1,1)$. 
As Assumption \ref{assumption:W} does not hold, we use the relaxation scheme in Section \ref{sec:rectangular}. We apply Algorithm \ref{alg:1} to \PMPODElambdaunder{d^0} and \PMPODElambdaover{d^0} for different values of the relaxation parameter $\lambda$. Theorem \ref{thm:hull_F:rect_lambda} ensures that these approximations inner- and outer-approximate the true convex hulls of the reachable sets, and that these approximations converge as the relaxation parameter $\lambda$ increases. Indeed, the results in Figure \ref{fig:dubins} (left) show that the approximations converge as $\lambda$ increases.

\subsubsection{Non-invertible $g(t,x)$}
Let $G^\top=\AverageSmallMatrix{
1&0&0\\0&0&1
}$ and $\W=B(0,10^{-2})\subset\R^2$. As Assumption \ref{assumption:g} does not hold, we use the relaxation scheme in Section \ref{sec:noninvertible}. We apply Algorithm \ref{alg:1} to \PMPODEepsilonfullrank{d^0} using the approximation of $\W$ in Example \ref{example:gauss_maps:not_full_rank} and $g_3(t,x)=(0,1,0)$. 
Since Assumptions \ref{assumption:Wbar:smooth_fulldim} and \ref{assumption:g_bar_epsilon} are satisfied, Theorem \ref{thm:hull_F:not_full_rank} guarantees that the approximations get closer to the true convex hulls of the reachable sets as the relaxation parameter $\epsilon$ decreases. Indeed, the results in Figure \ref{fig:dubins} (right) show that the approximations converge as $\epsilon$ decreases.

\begin{figure*}[!t]
\centering
\begin{minipage}{1\textwidth}
\begin{minipage}{0.67\textwidth}
\includegraphics[width=0.49\textwidth]{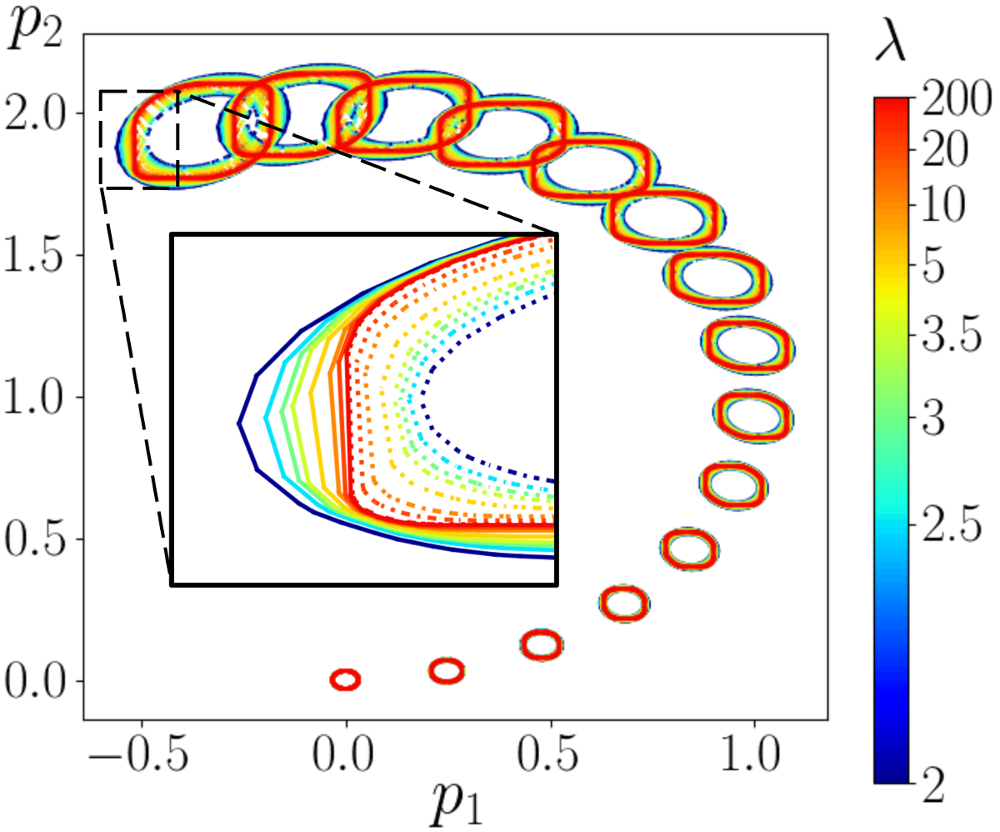}
\hspace{0.01\linewidth}
\includegraphics[width=0.49\linewidth]{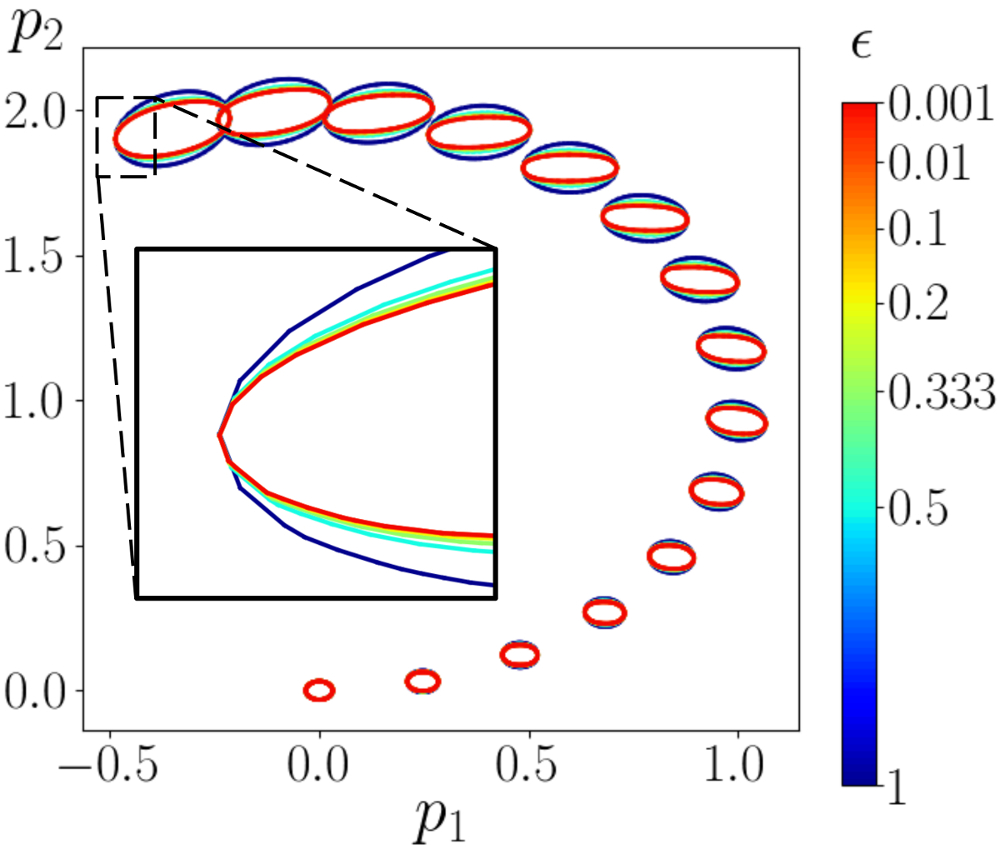}
\end{minipage}
\hspace{0.01\textwidth}
\begin{minipage}{0.3\textwidth}
\includegraphics[width=1\linewidth]{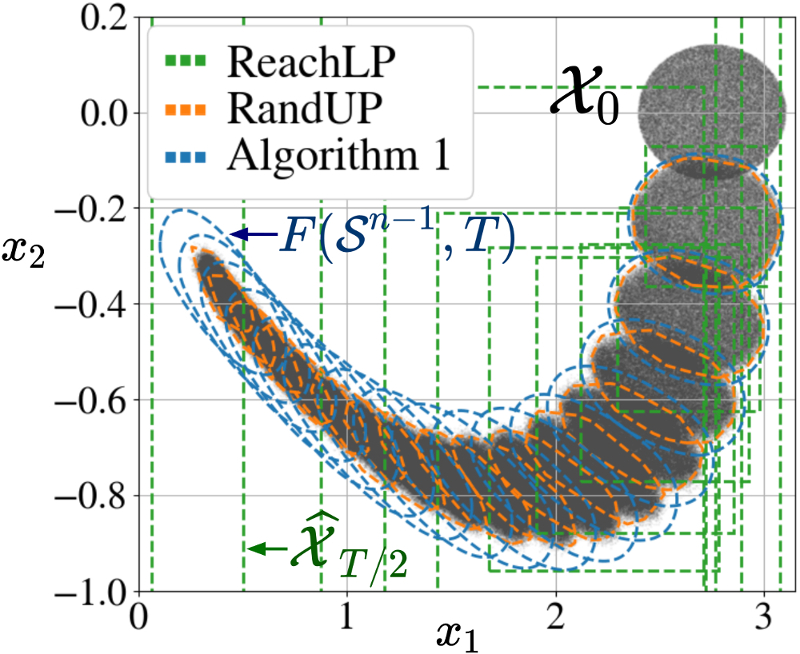}
\end{minipage}
\end{minipage}
\\
\begin{minipage}{1\textwidth}
\begin{minipage}{0.67\textwidth}
\caption{Output of Algorithm \ref{alg:1} on Dubins car dynamics (Section \ref{sec:numerical:dubins}) for different relaxation parameters. 
\textit{Left} (rectangular disturbance set): inner- (dashed lines) and outer- (full lines) approximations using relaxation scheme in Theorem \ref{thm:hull_F:rect_lambda}. 
\textit{Right} (non-invertible $g(t,x)$): outer-approximations using relaxation scheme in Theorem \ref{thm:hull_F:not_full_rank}.}
\label{fig:dubins}
\end{minipage}
\hspace{0.01\textwidth}
\begin{minipage}{0.3\textwidth}
\caption{Neural feedback loop analysis. Estimates of the reachable set convex hulls estimates ($M=10^3$) and Monte-Carlo samples (in gray).}
\label{fig:neural_net}
\end{minipage}
\end{minipage}
\vspace{-6mm}
\end{figure*}

\subsection{Neural feedback loop analysis}\label{sec:numerical:nn}
\begin{figure}[!t]
\centering
\includegraphics[width=0.515\linewidth,trim=5 5 5 5, clip]{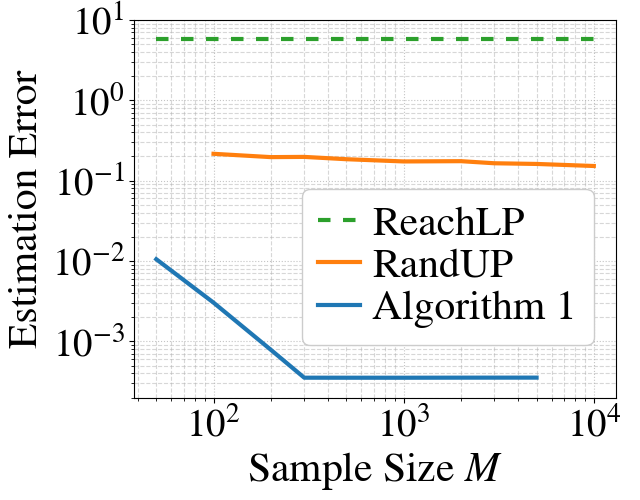}
\includegraphics[width=0.47\linewidth,trim=35 5 5 5, clip]{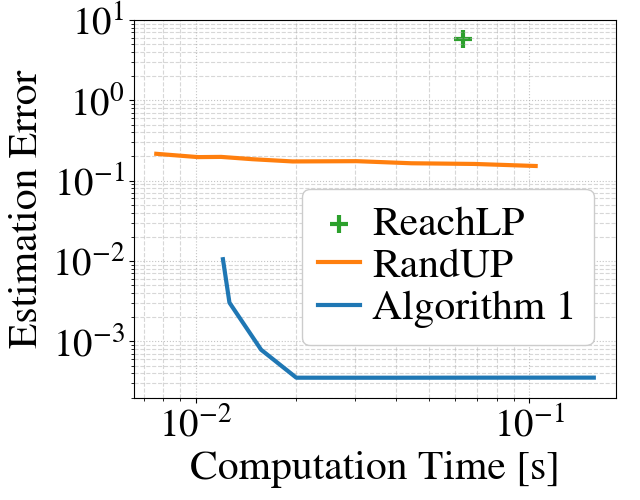}
\caption{Neural feedback loop analysis. Estimation error $d_H(\hull(\X_T),\hull(\{x_i\}_{i=1}^M)$ vs sample size and computation time.}
\label{fig:neural_net:comp_times}
\end{figure}

Consider the system $\dot{x}(t)=Ax(t)+B\pi(x(t))+w(t)$, where $\pi$ is a neural network and $n=2$. The dynamics parameters $(A,B,\pi)$ are as in \cite{LewBonalliEtAl2023} and \cite[Sec. VIII.A-C]{Everett21_journal}. 
The sets of initial conditions and disturbances $\X_0$ and $\W$ are ellipsoidal sets, %
see Appendix \ref{appendix:numerical:nn} for  details.

\subsubsection{Validating Theorem \ref{thm:hull_F}}
We run Algorithm \ref{alg:1} with $M=10^3$ samples of $d^0$ evenly covering the circle $\cS^{n-1}$. %
Then, we uniformly sample $10^5$ disturbances $w(t)\in\W$ at each timestep and evaluate the corresponding trajectories. We verify that all resulting trajectories are within the convex hulls computed by Algorithm \ref{alg:1}, which empirically validates Theorem \ref{thm:hull_F}. Thus, using Algorithm \ref{alg:1}, it suffices to sample initial directions $d^0\in\cS^{n-1}$ to reconstruct the convex hulls of the reachable sets.

\subsubsection{Comparisons}
We consider a first baseline that randomly samples disturbances $w(t)\in\W$ at each timestep and returns the convex hulls of closed-loop trajectories (\texttt{RandUP} \cite{LewJansonEtAl2022}). As ground truth, we use Algorithm \ref{alg:1} with a very large number of samples ($10^4$), which is justified by Theorem \ref{thm:hull_F}. We report results in Figures \ref{fig:neural_net} and \ref{fig:neural_net:comp_times}. Algorithm \ref{alg:1} returns estimates that are orders of magnitude more accurate than the baseline's estimates. Given a desired accuracy, Algorithm \ref{alg:1} is thus orders of magnitude faster than the baseline. This difference is a direct consequence of Theorem \ref{thm:hull_F}: only initial directions $d^0$ on the sphere $\cS^{n-1}$ need to be sampled using Algorithm \ref{alg:1}. In contrast, the baseline requires sampling a much larger number of variables that increases with the number of prediction timesteps, resulting in significantly worse sample complexity (see \cite{LewJansonEtAl2022,LewBonalliJansonPavone2023} for error bounds) and trajectories that are far from the reachable set boundaries. 
Recall that Algorithm \ref{alg:1} returns under-approximations of the convex hulls of
the reachable sets, since any extremal trajectory $x_{d^0}$
solving \PMPODE{d^0} is contained in the true reachable
sets. These simulations show that naive Monte-Carlo
estimates poorly approximate the convex hulls of the
reachable sets for this problem.

In Figure \ref{fig:neural_net} and \ref{fig:neural_net:comp_times}, we also report the over-approximations from a formal method (\texttt{ReachLP} \cite{Everett21_journal}) that significantly over-estimate the true reachable sets and is unable to capture the closed-loop stability of the system. In contrast, Algorithm \ref{alg:1} returns accurate approximations using a small number of samples. \rev{Other reachability methods may return more accurate approximations for this type of problems \cite{ManzanasLopez2023}, albeit potentially at the expense of additional computation time.}

\subsection{Robust MPC for attitude control of a spacecraft}\label{sec:numerical:spacecraft}
We design an attitude controller for a spacecraft with state $x=(q,\omega)\in\R^7$, control $u\in\R^3$, and  dynamics %
\begin{subequations}\label{eq:spacecraft:dynamics}
\begin{align}
\dot{q}(t)
&=
\Omega(\omega(t))q(t),
\\
\dot{\omega}(t) &= J^{-1}(u(t)-S(\omega(t))J\omega(t) + w(t)),
\label{eq:spacecraft:dynamics:omega}
\end{align}
\end{subequations}
with $x(0)=x^0=(q^0,\omega^0)$, $w(t)\in \W=B(0,10^{-2})$, inertia matrix $J=\textrm{diag}(5,2,1)$, and matrices $(\Omega,S)(\omega)$ defined in \cite{Leeman2023,LewBonalliEtAl2023}. We constrain $\omega(t)$ and $u(t)$ as
\begin{equation}\label{eq:spacecraft:constraints}
\|\omega(t)\|_\infty\leq 0.1,
\quad 
\|u(t)\|_\infty\leq 0.1,
\quad 
t\in[0,T].
\end{equation}
We consider feedback controls parameterized as $u(t)=\bar{u}(t)+K\omega(t)$, where $\bar{u}\in L^\infty([0,T],\R^3)$  is an open-loop control and 
$K=-\textrm{diag}(5,2,1)$ is a feedback gain. %
To enforce \eqref{eq:spacecraft:constraints}, we define the reachable set
\begin{equation}
\cR_t(\bar{u})=
\left\{\omega_w(\bar{u},t):\  \begin{aligned}
\omega_w(\bar{u})\ \text{solves}\ \eqref{eq:spacecraft:dynamics:omega}
\\[-1mm]
u(t)=\bar{u}(t)+K\omega(t)
\\[-1mm]
\hfill w\in L^\infty([0,T],\W)
\end{aligned}
\right\}
\end{equation}
for any $\bar{u}\in L^\infty([0,T],\R^3)$ and $t\in[0,T]$. 
The convex hulls of the reachable sets $\hull(\cR_t(\bar{u}))$ can be estimated using Algorithm \ref{alg:1}, where \PMPODE{d^0} is defined using \eqref{eq:spacecraft:dynamics:omega} with $u(t)=\bar{u}(t)+K\omega(t)$ and $\X_0=\{x^0\}$ is a singleton. %
Thus, given $M$ samples $d^i\in\cS^{n-1}$ and $\epsilon_t>0$ large-enough, the constraints  \eqref{eq:spacecraft:constraints} can be approximated by 
\begin{subequations}
\label{eq:spacecraft:constraints:approx}
\begin{gather}
-0.1+\epsilon_t\leq \omega^i(\bar{u},t)_j
\leq 0.1-\epsilon_t,
\\
-0.1+\epsilon_t\leq 
(\bar{u}(t)+ K\omega^i(\bar{u},t))_j
\leq 0.1-\epsilon_t,
\label{eq:spacecraft:constraints:approx:controls}
\end{gather}
\end{subequations}
for all $j=1,2,3$, $i=1,\dots,M$, $t\in[0,T]$. 
The conservatism of \eqref{eq:spacecraft:constraints:approx} follows from the convexity of  \eqref{eq:spacecraft:constraints} and Corollary \ref{cor:error_bound} or Theorem \ref{thm:error_bound}, see \cite[Corollary 5.5]{LewBonalliJansonPavone2023}. 
Given a reference $x_\text{r}=(1,0,\dots,0)$ and $(Q,R)=(10I_7,I_3)$, we define the robust control problem %
\customlabel{OCPMPC}{$\textbf{OCP}(x^0)$}:
\begin{equation*}
\begin{aligned}
\inf_{\bar{u}} \ 
  &\int_0^T((x_{\bar{u}}(t)-x_\text{r})^\top Q(x_{\bar{u}}(t)-x_\text{r}) + \bar{u}(t)^\top R \bar{u}(t))\dd t
\\
\textrm{s.t.} \  
& \,
  x_{\bar{u}} \text{ solves }\eqref{eq:spacecraft:dynamics} \text{ with }w=0,
\text{ and }
\bar{u}\text{ satisfies }\eqref{eq:spacecraft:constraints:approx}.
\end{aligned}
\end{equation*}
By recursively solving \OCPMPC{x^0} and applying the computed control inputs, we obtain the  receding horizon robust MPC controller in Algorithm \ref{alg:mpc}. We use $M=50$ samples of $d^0$ and the error bounds $\epsilon_t$ in Theorem \ref{thm:error_bound}. We solve \OCPMPC{x^0} using a standard direct method based on sequential convex programming (SCP). %
We refer to Appendix \ref{appendix:numerical:spacecraft} and the open-source code for further details.

\begin{figure}[!t]
 \centering
\begin{tcolorbox}[
title={\textit{Alg. \customlabel{alg:mpc}{3}(Robust model predictive control (MPC))}\phantom{${}^1$}},
colbacktitle=ForestGreen!10, 
coltitle=black,
boxrule=3pt,
colframe=ForestGreen!10, %
colback=ForestGreen!4,
enhanced, boxsep=0pt, left=6pt, right=6pt]
 \vspace{-4mm}
 \begin{algorithm}[H]
 \textbf{Input}: $M$ initial directions $d^0\subset\cS^{n-1}$, relaxation constants $\epsilon_t>0$, initial state $x(0)$
 \begin{algorithmic}
 \ForAll {$k=0,1,2\dots$}
 \State $\bar{u}^k\gets\textrm{Solve}(\textbf{OCP}(x(k\Delta t))$
 \State Apply $u^k(t)=\bar{u}^k(t)+K\omega(t)$ for $t\in[0,\Delta t]$
\EndFor
 \end{algorithmic}
 \end{algorithm}
 \vspace{-8mm}
 \end{tcolorbox}
\vspace{-6mm}
 \end{figure}

\subsubsection{MPC results}
We evaluate the controller in $100$ experiments with uniformly-sampled disturbances and initial states. Results in Figure \ref{fig:spacecraft:mpc} show that despite disturbances, the system converges to the reference and the constraints \eqref{eq:spacecraft:constraints} are always satisfied. %
The optimization problem is always feasible in these experiments. %
 We observe that the error bounds from Theorem \ref{thm:error_bound} introduce reasonable conservatism. By increasing the sample size $M$, this conservatism can be made arbitrarily small.

\begin{figure}[!t]
\centering
\begin{minipage}{0.475\linewidth}
\includegraphics[width=1\linewidth]{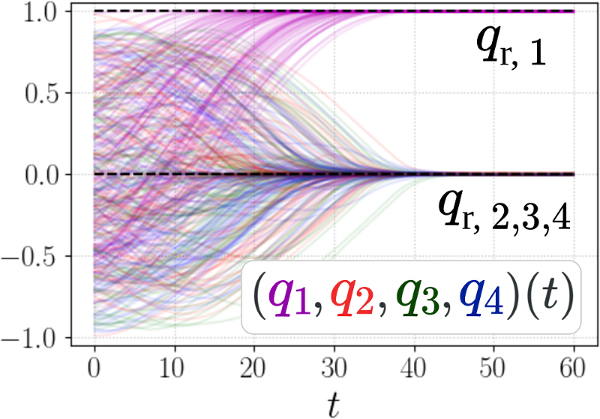}
    \centering  
\caption{$100$ closed-loop trajectories using robust MPC (Algorithm \ref{alg:mpc}) to stabilize the attitude of the spacecraft from different initial conditions under external disturbances.}
\label{fig:spacecraft:mpc}
\end{minipage}
\begin{minipage}{0.5\linewidth}
\includegraphics[width=1\linewidth]{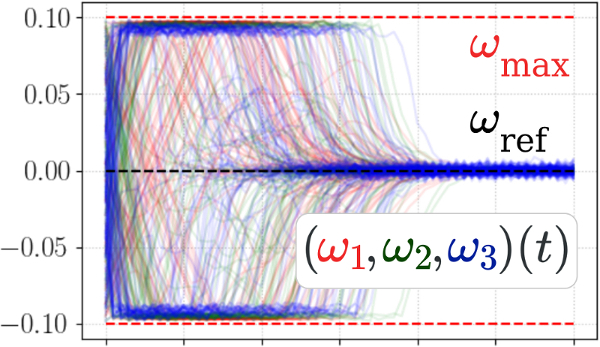}
\\[1mm]
\includegraphics[width=1\linewidth]{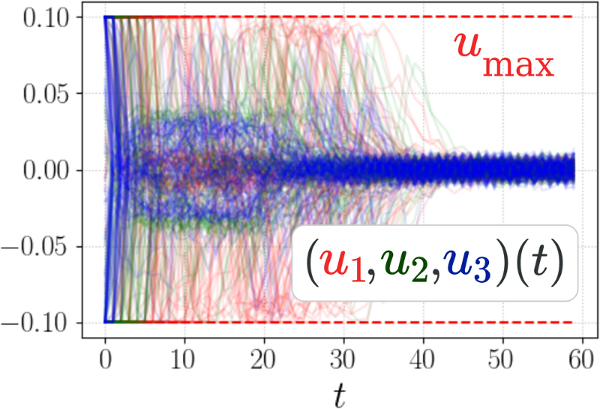}
\end{minipage}
\vspace{-6mm}
\end{figure}

As is common in MPC, in Algorithm \ref{alg:mpc}, we warm-start the optimization using the previously computed solution and only perform a single SCP iteration per timestep, yielding a replanning rate with MPC of approximately $20$\textrm{Hz} with our Python implementation. 
We report solver statistics and computation times from a zero-initial guess in the appendix in Figures \ref{fig:spacecraft:comp_times} and \ref{fig:spacecraft:comp_times_Ms}. We observe that a few SCP iterations suffice to reach accurate solutions. 
 Computation time roughly scales linearly with the sample size (most computation time is spent  evaluating \eqref{eq:spacecraft:constraints:approx}) and could be reduced via parallelization on a GPU.

\subsubsection{Comparisons with other reachability methods}
We compare the reachable set convex hull estimates from Algorithm \ref{alg:1} with those from two other standard methods. The first baseline is a sampling-based method (\texttt{RandUP}  \cite{LewPavone2020}) that estimates the convex hulls $\hull(\X_t)$ with the convex hulls of trajectories from \eqref{eq:dynamics:discrete-time} with samples of $w(k\Delta t)$. The second standard baseline propagates uncertainty from the disturbances using a linear model of \eqref{eq:dynamics:discrete-time} and bounds the approximation error using the Lipschitz constant of the Jacobian $\nabla_x\bar{f}(x,u)$. %

Given a control trajectory $\bar{u}$ solving \OCPMPC{x^0}, we present reachable set estimates  in Figure \ref{fig:spacecraft:reach_comparison}. %
First, the Lipschitz-based and the naive sampling-based baselines are the fastest (with runtimes at $35\mu\textit{s}$ and $150\mu\textit{s}$, respectively), followed by Algorithm \ref{alg:1} ($350\mu\textit{s}$). 
However, the over-approximations of the reachable sets from the Lipschitz-based method are significantly more conservative than those from Algorithm \ref{alg:1}. A controller using the reachable set estimates from this baseline would deem $\bar{u}$ to potentially violate constraints and would thus be more conservative than the proposed robust MPC approach. 
Also, the naive sampling-based baseline significantly under-estimates the true convex hulls. %
One can show that this baseline performs worse as the discretization is refined, see also Section \ref{sec:numerical:nn}. In contrast, Algorithm \ref{alg:1} is derived in continuous time so its sample complexity is independent of the discretization of the dynamics. 
Since Algorithm \ref{alg:1} only samples on the $(n-1)$-dimensional sphere $\cS^{n-1}$, it is more efficient and its precision only depends on the accuracy of the discretization of \PMPODE{d^0}.

\begin{figure}[!t]
\centering
\includegraphics[width=1\linewidth,trim=10 10 0 5, clip]{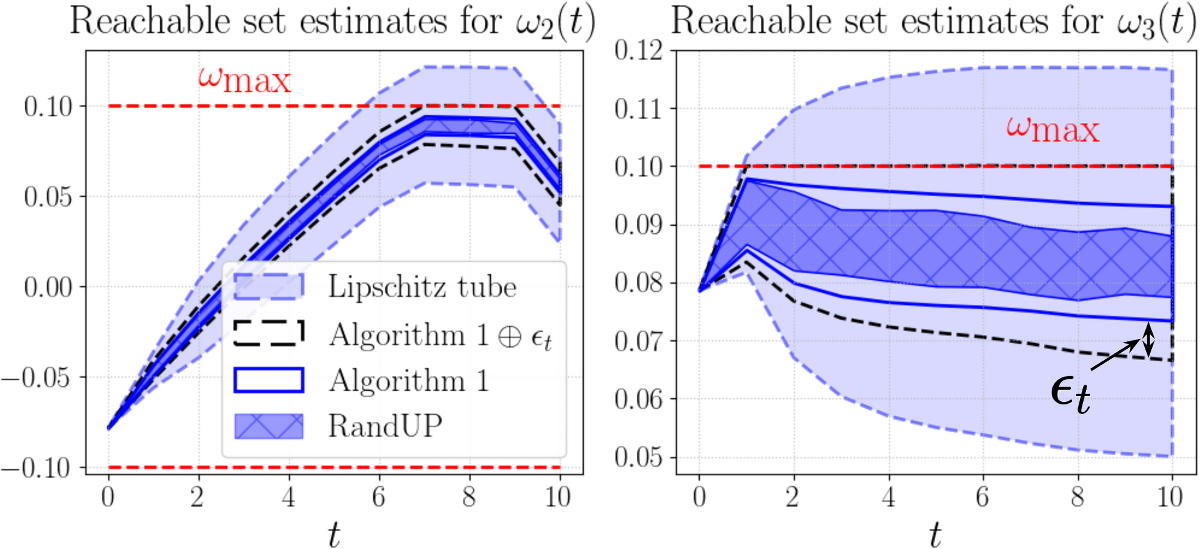}
 \caption{Reachability comparison: convex hull reachable sets estimates computed with Algorithm \ref{alg:1} and with two baselines.}
\label{fig:spacecraft:reach_comparison}
\vspace{-4mm}
\end{figure}

\section{Conclusion}\label{sec:conclusion}
We showed that estimating the convex hulls of reachable sets of nonlinear systems with disturbances and uncertain initial conditions is equivalent to studying the solutions of an ODE with initial conditions on the sphere. 
This result is a significantly simpler finite-dimensional characterization of the convex hulls of reachable sets \rev{that could inform the design of efficient reachability analysis algorithms for nonlinear systems.}

\rev{Algorithm \ref{alg:1} 
has two main limitations. 
First, the accuracy of sampling-based techniques decreases as the number of uncertain variables increases. Thanks to our characterization result, the sample space is only of dimension 
$(n-1)$ as opposed to an infinite-dimensional space of disturbances. 
However, obtaining provably-accurate approximations for high-dimensional systems in reasonable computation time remains difficult. 
This limitation is unfortunately shared by other reachability analysis algorithms for nonlinear systems. 
It would be interesting to develop methods to bias sampling to get accurate approximations with fewer samples, e.g., using  adversarial sampling \cite{LewPavone2020}, or use additional properties of the dynamics to design a method that out-performs sampling-based-only algorithms. 
Second, convex hull approximations of non-convex reachable sets may be conservative for some systems. Given additional computation time, this limitation could be addressed by splitting the sample space $\cS^{n-1}$ into distinct regions, running Algorithm \ref{alg:1} on each region, and approximating the reachable sets with the non-convex union of the outputs.}

\ifarxiv
    
    \section*{References}
    \bibliographystyle{IEEEtran}
    \bibliography{ASL_papers,main}

\fi

\appendix

\subsection{The linear case with ellipsoidal uncertainties}\label{apdx:linear_case}
\rev{To gain further intuition, we describe how the results specialize to the problem setting with
linear dynamics
$$\dot{x}(t)=A(t)x(t)+B(t)w(t)$$ and ellipsoidal uncertainty sets 
 $\mathcal{X}_0=\mathcal{E}(\mu,Q)=\{x\in\R^n:(x-\mu)^\top Q^{-1}(x-\mu)\leq 1\}$ (see Example \ref{example:gauss_maps}) and $\W=\mathcal{E}(\bar{w},W)$. In this setting, \PMPODE{d^0} simplifies to
\begin{align*}
\dot{x}(t)&=A(t)x(t)+B(t)w(t),
\qquad\qquad\quad  t\in[0,T],
\nonumber
\\
\dot{p}(t)&=
-A(t)^\top p(t),
\\
w(t)&=
(n^{\partial\W})^{-1}\left(\frac{A(t){}^\top p(t)}{\|A(t){}^\top p(t)\|}\right)
\\
&= \bar{w}+WA(t){}^\top p(t)/\sqrt{p(t){}^\top A(t) W A(t){}^\top p(t)}
\\
x(0)&=
(n^{\partial\X_0})^{-1}(d^0)
\\
&=\mu+Qd^0/\sqrt{d^0{}^\top Q d^0},
\\
p(0)&=d^0.
\end{align*}
We can also draw a direct connection with the work of 
Kurzhanski and Varaiya in 
\cite{Kurzhanski2000}.}
\begin{lemma}[3, Equation 30] 
\rev{Consider  the linear system 
 $\dot{x}(t)=A(t)x(t)+B(t)w(t)$ with  $x(0)\in\mathcal{X}_0$ and $w(t)\in\mathcal{W}$, where $B(t)$ is always invertible, 
 $\mathcal{X}_0=\mathcal{E}(\mu,Q)$ 
 is a nondegenerate ellipsoid, and $\partial\mathcal{W}$ is an ovaloid. Given a direction $d\in\mathcal{S}^{n-1}$, the initial state $x(0)\in\mathcal{X}_0$ that corresponds to the reachable state $x(T)\in\mathcal{X}_T$ that is the furthest in the direction $d$ is given by $$
 x(0)=\mu+Qp(0)/\sqrt{p(0)^\top Qp(0)},
 $$ where $p(0)=M(T)^{-1}d$ with $M(t)$ the solution to the matrix ODE $\dot{M}(t)=-A(t)^\top M(t)$ with $M(0)=I$.}
\end{lemma}
\begin{proof}
    \rev{By formulating OCP${}_d$ and BVP${}_d$ as in Section \ref{sec:structure_proof}.A-B, we know that the adjoint vector satisfies $p(T)=d$ and $\dot{p}(t)=-A(t)^\top p(t)$, so its initial value is $p(0)=M(T)^{-1}p(T)=M(T)^{-1}d$ with $M(t)$ solving the matrix ODE $\dot{M}(t)=-A(t)^\top M(t)$ with $M(0)=I$. 
    Thus, using (27), the initial state is $x(0)=n^{\partial\mathcal{X}_0}(p(0))=\mu+Qp(0)/\sqrt{p(0)^\top Qp(0)}$. }
\end{proof}

\rev{This result is precisely \cite[Equation 30]{Kurzhanski2000}. %
In the nonlinear case however, it is difficult to derive a closed-form expression for $p(0)$ as a function of $p(T)=d$, and one has to solve \BVP{d} instead to find $x(0)$. This is because the evolution of $p(t)$ depends on the state trajectory and on the value of $x(T)$ that is unknown. If instead, the objective is to reconstruct the convex hull of the entire reachable set, then Theorem \ref{thm:hull_F} states that it suffices to integrate the coupled \PMPODE{d^0} forward in time from $(x(0),p(0))=((n^{\partial\X_0})^{-1}(d^0),d^0)$ for all $d^0\in\cS^{n-1}$. See also Section \ref{sec:structure_proof:discussion} for further comments.}

\section{Additional results and proofs}
\subsection{Inverse Gauss map of level set ovaloids}\label{apdx:sec:gauss_map_level_set:inv}
We prove \eqref{eq:gauss_map_level_set:inv} in Example \ref{example:gauss_maps}.
\begin{lemma}[Inverse of Gauss maps of sublevel sets]\label{lem:gauss_map_level_set:inv}
Let $h:\R^n\to\R$ be smooth and $C=\{x\in\R^n:h(x)\leq 1\}$ be a compact set whose boundary $\partial C$ is an ovaloid. Assume that
\begin{equation}
\label{eq:gauss_map_level_set:inv:condition}
h\left(\nabla h^{-1}(n(x))\right)=1/\|\nabla h(x)\|^2
\ \ \text{for all}\  x\in\partial C.
\end{equation}
Then, the inverse Gauss map of $\partial C$ is given as
\begin{equation}\tag{\ref{eq:gauss_map_level_set:inv}}
n^{-1}(d)=\nabla h^{-1}\left(\frac{d}{\sqrt{h(\nabla h^{-1}(d)}}\right).
\end{equation}
\end{lemma}
\begin{remark}[On the condition in \eqref{eq:gauss_map_level_set:inv:condition}]\label{remark:gauss_map_level_set:inv:condition}
The condition in \eqref{eq:gauss_map_level_set:inv:condition} is necessary. Indeed, one verifies that the level set function $h(x)=\|x\|^4$ violates both \eqref{eq:gauss_map_level_set:inv:condition} and  \eqref{eq:gauss_map_level_set:inv}, although its sublevel set $\{x\in\R^n:h(x)\leq 1\}$ is the ball $B(0,1)$. One verifies that 
$$
\nabla h^{-1}\left(d / \sqrt{h(\nabla h^{-1}(d)}\right)=
\frac{d}{2^{1/9}},
$$
which is different than the inverse Gauss map $n^{-1}(d)=d$ of the ball $B(0,1)$.  
The condition in \eqref{eq:gauss_map_level_set:inv:condition} holds for the three sublevel sets used in this work (see  Example \ref{example:gauss_maps}) because their level set functions $h(x)$ behave approximately like the squared norm function $h(x)=\|x\|^2$. 
\end{remark}
\begin{proof}
The Gauss map $n:\partial C\to\cS^{n-1}, x\mapsto\nabla h(x)/\|\nabla h(x)\|$ is well-defined under the assumptions. Thus, the map $\nabla h$ and its inverse are well-defined. First,
\begin{align*}
n(n^{-1}(d))
&=
\nabla h(n^{-1}(d))/\|\nabla h(n^{-1}(d))\|
\mathop{=}^\eqref{eq:gauss_map_level_set:inv}d
\end{align*}  for all $d\in\cS^{n-1}$. Second, for any $x\in\partial C$,
\begin{align*}
n^{-1}(n(x))
&=
\nabla h^{-1}\left(\frac{n(x)}{\sqrt{h(\nabla h^{-1}(n(x))}}\right)
\\
&\mathop{=}^\eqref{eq:gauss_map_level_set:inv:condition}
\nabla h^{-1}\left(
\frac{\nabla h(x)/\|\nabla h(x)\|}{\sqrt{1/\|\nabla h(x)\|^2}}
\right)=x,
\end{align*}
so the map $n^{-1}(\cdot)$ in \eqref{eq:gauss_map_level_set:inv} is a well-defined inverse.
\end{proof}

\vspace{3mm}

\begin{tcolorbox}[
title={\textit{Example \ref{example:gauss_maps} revisited}},
colbacktitle=gray!20, 
coltitle=black,
breakable,
boxrule=3pt,
colframe=gray!20, %
boxsep=0pt, before skip=0pt, after skip=0pt]
The Gauss maps of the boundaries of common ovaloids are given in Example \ref{example:gauss_maps}. Here, we show that \eqref{eq:gauss_map_level_set:inv:condition} holds for these sets. Thus, by applying Lemma \ref{lem:gauss_map_level_set:inv}, the inverse Gauss map of these sets can be computed using \eqref{eq:gauss_map_level_set:inv}.

1) \textit{Ball}: Let $h(x)=\|x-\bar{x}\|^2/r^2$, so $C=B(\bar{x},r)$. Then, $\nabla h(x)=\frac{2}{r^2}(x-\bar{x})$, so $\nabla h^{-1}(y)=\bar{x}+y\frac{r^2}{2}$, so \eqref{eq:gauss_map_level_set:inv:condition} is easily verified.

2) \textit{Ellipsoid}: Let $h(x)=(x-\bar{x})^\top Q^{-1}(x-\bar{x})$, so $C=\mathcal{E}$. Then, $\nabla h(x)=2Q^{-1}(x-\bar{x})$, so $\nabla h^{-1}(y)=\bar{x}+\frac{1}{2}Qy$ and $n(x)=\frac{Q^{-1}(x-\bar{x})}{\|Q^{-1}(x-\bar{x})\|}$. Then, $\nabla h^{-1}(n(x))=\bar{x}+\frac{1}{2}\frac{(x-\bar{x})}{\|Q^{-1}(x-\bar{x})\|}$, so that
\begin{align*}
h(\nabla h^{-1}(n(x)))
&=
\frac{1}{4}\frac{(x-\bar{x})^\top Q^{-1}(x-\bar{x})}{\|Q^{-1}(x-\bar{x})\|^2}
\\
&\hspace{-1cm}=
\frac{1}{4}\frac{1}{\|Q^{-1}(x-\bar{x})\|^2}
=
\frac{1}{\|\nabla h(x)\|^2},
\end{align*}
as $h(x)=1$ for all $x\in\partial C$. Thus, \eqref{eq:gauss_map_level_set:inv:condition} is verified.

3) \textit{$\lambda$-balls}: For $\lambda>1$, let $h(x)=\|x-\bar{x}\odot\delta\bar{x}^{-1}\|_\lambda^2$, so $C=\eqref{eq:Clambda_under}$. Then, $\nabla h(x)=2\frac{(x-\bar{x})\odot|x-\bar{x}|^{\lambda-2}\odot\delta\bar{x}^{-\lambda}}{\|(x-\bar{x})\odot\delta\bar{x}^{-1}\|_\lambda^{\lambda-2}}$, so 
$\nabla h^{-1}(y)
=
\bar{x}+\frac{1}{2}\||y\odot\delta\bar{x}|^{\frac{1}{\lambda-1}}\|_\lambda^{\lambda-2}y^{\frac{1}{\lambda-1}}\odot\delta\bar{x}^{\frac{\lambda}{\lambda-1}}$. Then, with
\begin{equation}\tag{\ref{eq:gauss_map:Clambda}}
n(x)
=
\frac{(x-\bar{x})\odot|x-\bar{x}|^{\lambda-2}\odot\delta\bar{x}^{-\lambda}}{\||x-\bar{x}|^{\lambda-1}\odot\delta\bar{x}^{-\lambda}\|},
\end{equation}
one verifies that
\begin{align*}
\nabla h^{-1}(n(x)))
&=
\bar{x}+\frac{1}{2}\frac{\|(x-\bar{x})\odot\delta\bar{x}^{-1}\|^{\lambda-2}}{\||x-\bar{x}|^{\lambda-1}\odot\delta\bar{x}^{-\lambda}\|}
\big(
\\
&\quad(x-\bar{x})\odot|x-\bar{x}|\big)^{\frac{1}{\lambda-1}}.
\end{align*}
By computing $h(\nabla h^{-1}(n(x))))$ and $\|\nabla h(x)\|^2$, one verifies that \eqref{eq:gauss_map_level_set:inv:condition} holds.
\end{tcolorbox}

\subsection{Proof of Lemma \ref{lem:Clambda}}\label{sec:lem:Clambda:proof}
To prove Lemma \ref{lem:Clambda}, we use the following result.

\begin{lemma}[Hausdorff convergence of sub-level sets]\label{lem:sublevel_set_hausdoorff_to_zero}
Let $K\subset\R^n$ be a compact set, $h,h_\lambda:\R^n\to\R$ with $\lambda\geq 0$ be continuous maps, and define the compact sets
$$
C=\{x\in K:h(x)\leq 0\}, \ \
C_\lambda=\{x\in K:h_\lambda(x)\leq 0\}.
$$
Assume that 
\begin{itemize}
\item %
  $\sup_{x\in K}|h_\lambda(x)-h(x)|\to 0\text{ as }\lambda\to\infty$.
\item %
  For any $x\in C$ and $\epsilon>0$, there exists $z\in C$ such that $\|x-z\|\leq\epsilon$ and $h(z)<0$.\footnote{Only assuming the uniform convergence of the sub-level set functions is not sufficient \cite{Camilli1999} (as a counter-example, take $h(x)=0$ and $h_\lambda(x)=1/\lambda$). Assuming that the boundary of $C$ is smooth (e.g., $h(x)=\|x\|-R$ so $C$ is a ball of radius $R>0$) suffices to ensure the satisfaction of the interior point condition, albeit it is more conservative. For example, rectangular sets (with $h(x)=\|x\|_\infty-R$) do not have a smooth boundary but satisfy the interior point condition.}
\end{itemize}
Then,
\begin{equation}\label{eq:sublevel_set:hausdorff_to_zero}
d_H(C_\lambda,C)\to 0\text{ as }\lambda\to\infty.
\end{equation}
\end{lemma}

\textit{Proof of Lemma \ref{lem:sublevel_set_hausdoorff_to_zero}:} 
We proceed in two steps. 

First, we show that
\begin{equation}\label{eq:sublevel_set:hausdorff_to_zero:1}
\sup_{x\in C}\Inf_{y\in C_\lambda}\|x-y\|\to 0\text{ as }\lambda\to\infty.
\end{equation}
Let $\epsilon>0$. We claim that $\exists\lambda>0$ such that $\forall\bar\lambda\geq\lambda$,
\begin{equation}\label{eq:sublevel_set:hausdorff_to_zero:1:epsilon}
\sup_{x\in C}\Inf_{y\in C_{\bar\lambda}}\|x-y\|\leq\epsilon.
\end{equation}
Let $x\in C$ be arbitrary. Then, 
$\exists z\in C$ such that $\|x-z\|\leq\epsilon$ with $h(z)=-\delta<0$ for some $\delta>0$. 
Moreover, %
$\exists\lambda>0$ such that $\forall\bar\lambda\geq\lambda$, $\sup_{x\in K}\|h_{\bar\lambda}(x)-h(x)\|\leq\delta/2$. Thus,
\begin{align*}
h_{\bar\lambda}(z)&=h_{\bar\lambda}(z)-h(z)+h(z)
\\
&\leq
\sup_{x\in K}
\|h_{\bar\lambda}(x)-h(x)\|+h(z)
=-\delta/2
<0,
\end{align*}
so that $z\in C_{\bar\lambda}$. We obtain that $\Inf_{y\in C_{\bar\lambda}}\|x-y\|\leq\epsilon$, which implies that \eqref{eq:sublevel_set:hausdorff_to_zero:1:epsilon}, and thus \eqref{eq:sublevel_set:hausdorff_to_zero:1}.

Second, we prove that 
\begin{equation}
\label{eq:sublevel_set:hausdorff_to_zero:2}
\sup_{x\in C_\lambda}\Inf_{y\in C}\|x-y\|\to 0\text{ as }\lambda\to\infty.
\end{equation}
Let $\epsilon>0$. We claim that $\exists\lambda>0$ such that $\forall\bar\lambda\geq\lambda$,
\begin{equation}
\label{eq:sublevel_set:hausdorff_to_zero:2:epsilon}
\sup_{x\in C_{\bar\lambda}}\Inf_{y\in C}\|x-y\|
\leq\epsilon.
\end{equation}
By contradiction, assume that $\forall\lambda>0$, $\exists\bar\lambda\geq\lambda$ such that 
$$
\sup_{x\in C_{\bar\lambda}}\Inf_{y\in C}\|x-y\|
>\epsilon.
$$
Then, there exists a sequence $(x_{\bar\lambda_\lambda})_{\lambda>0}$ such that $x_{\bar\lambda_\lambda}\in C_{\bar\lambda_\lambda}$ for all $\lambda>0$ and 
\begin{equation}\label{eq:sublevel_set:hausdorff_to_zero:2:epsilon:contradiction}
\Inf_{y\in C}\|x_{\bar\lambda_\lambda}-y\|
>\epsilon/2
\text{ for all }\lambda>0.
\end{equation}
Since $x_{\bar\lambda_\lambda}\in C_{\bar\lambda_\lambda}\subseteq K$ for all $\lambda>0$ and $K$ is compact, there exists a subsequence $(x_{{\bar\lambda_{\lambda_k}}})_{k\in\bN}$ and $x\in K$ such that $x_{{\bar\lambda_{\lambda_k}}}\to x$ as $k\to\infty$. Thus, $\exists N\in\bN$ such that
$$
\|x_{{\bar\lambda_{\lambda_k}}}-x\|
\leq\epsilon/2\text{ for all }k\geq N.
$$
Moreover, since $h$ is continuous and $h_{{\bar\lambda_{\lambda_k}}}(x_{{\bar\lambda_{\lambda_k}}})\leq 0$,
\begin{align*}
h(x)&=
h\Big(\lim_{k\to\infty}x_{{\bar\lambda_{\lambda_k}}}\Big)
=
\lim_{k\to\infty}h(x_{{\bar\lambda_{\lambda_k}}})
\\
&=
\lim_{k\to\infty}
\left(
h(x_{{\bar\lambda_{\lambda_k}}})-h_{{\bar\lambda_{\lambda_k}}}(x_{{\bar\lambda_{\lambda_k}}})
+
h_{{\bar\lambda_{\lambda_k}}}(x_{{\bar\lambda_{\lambda_k}}})
\right)
\\
&\leq
\lim_{k\to\infty}
\left(
\sup_{z\in K}
\|
h(z)-h_{{\bar\lambda_{\lambda_k}}}(z)\|
+
h_{{\bar\lambda_{\lambda_k}}}(x_{{\bar\lambda_{\lambda_k}}})
\right)
\\
&\leq
\lim_{k\to\infty}
\sup_{z\in K}
\|
h(z)-h_{{\bar\lambda_{\lambda_k}}}(z)\|
= 0,
\end{align*}
so $x\in C$.  Thus,
$$
\Inf_{y\in C}
\|x_{{\bar\lambda_{\lambda_k}}}-y\|
\leq
\|x_{{\bar\lambda_{\lambda_k}}}-x\|
\leq\epsilon/2\text{ for all }k\geq N,
$$
contradicting \eqref{eq:sublevel_set:hausdorff_to_zero:2:epsilon:contradiction}. Thus,  \eqref{eq:sublevel_set:hausdorff_to_zero:2:epsilon} holds, which implies \eqref{eq:sublevel_set:hausdorff_to_zero:2}. 

Combining \eqref{eq:sublevel_set:hausdorff_to_zero:1} and \eqref{eq:sublevel_set:hausdorff_to_zero:2}  gives \eqref{eq:sublevel_set:hausdorff_to_zero}.
\phantom{a} \hfill $\blacksquare$

\textit{Proof of Lemma \ref{lem:Clambda}:}
The first claims follow from the convexity of the map $x\mapsto \|x\|_\lambda$ for any $\lambda>1$, the continuity of the maps $h$ and $h_\lambda$ defined in \eqref{eq:C:h} and \eqref{eq:C:hlambda}, and the fact that $
\|x\|_\infty\leq\|x\|_\lambda\leq n^{\frac{1}{\lambda}}\|x\|_\infty\leq n^{\frac{1}{\lambda}}\|x\|_\lambda
$.

The third claims are shown in three steps. We prove these claims for $\partial\uwidehat{C_\lambda}$; the proof for $\partial\widehat{C_\lambda}$ is similar. 
First, 
$$
\partial \uwidehat{C_\lambda}=\{x\in\R^n:h_\lambda(x)=1\}.
$$
Second, $\partial\uwidehat{C_\lambda}$ is an $(n-1)$-dimensional submanifold, since $\nabla h_\lambda(x)$ is full rank for all $x\in\partial\uwidehat{C_\lambda}$ \cite[Theorem 5.12]{Lee2012}. Thus, the set $\partial\uwidehat{C_\lambda}$ has a well-defined unit-norm outward-pointing normal vector at any $x\in\partial\uwidehat{C_\lambda}$, denoted by $n^{\partial\uwidehat{C_\lambda}}(x)$, which is given by \cite[Chapter 8]{Lee2018}
\begin{equation*}
n^{\partial\uwidehat{C_\lambda}}(x)=
\nabla h_\lambda(x)/\|\nabla h_\lambda(x)\|
=\eqref{eq:gauss_map:Clambda}.
\end{equation*}
One checks that the map $n^{\partial\uwidehat{C_\lambda}}:\partial\uwidehat{C_\lambda}\to\cS^{n-1}$ is a well-defined diffeomorphism, and that its inverse $n^{\partial\uwidehat{C_\lambda}}:\cS^{n-1}\to\partial\uwidehat{C_\lambda}$ is given by \eqref{eq:gauss_map:Clambda_under:inverse}. 
Third, we show that $\partial\uwidehat{C_\lambda}$ has strictly positive curvature. Assuming that $\bar{x}_i=0$ and $\delta\bar{x}_i=1$ for all $i=1,\dots,n$ without loss of generality\footnote{The proof in the general case follows similarly after making the affine change of variable $x\mapsto\bar{x}+x\odot\delta\bar{x}^{\lambda/\lambda-1}$.}, 
the Jacobian of $n^{\partial\uwidehat{C_\lambda}}$ is
\begin{align*}
\nabla n^{\partial\uwidehat{C_\lambda}}(x)
&=
\bigg(\frac{\lambda-1}{\left\||x|^{\lambda-1}\right\|}
\text{diag}\left(|x|^{\lambda-2}\right)
-
\\
&\frac{\lambda-1}{\left\||x|^{\lambda-1}\right\|^3}\left(x\odot |x|^{\lambda-2}\right)\left(x\odot |x|^{2\lambda-4}\right)^\top\bigg).
\end{align*}
Since $n^{\partial\uwidehat{C_\lambda}}(x)\in N_x\uwidehat{C_\lambda}$, 
any tangent vector $v\in T_x\uwidehat{C_\lambda}$ satisfies $v^\top n^{\partial\uwidehat{C_\lambda}}(x)=0$. Thus, from $n^{\partial\uwidehat{C_\lambda}}(x)=\eqref{eq:gauss_map:Clambda}$, we obtain $v^\top(x\odot |x|^{\lambda-2})=0$. Thus, for all $v\in T_x\partial\uwidehat{C_\lambda}$,
$$
v^\top \nabla n^{\partial\uwidehat{C_\lambda}}(x) v 
= 
\frac{\lambda-1}{\left\||x|^{\lambda-1}\right\|}
\sum_{i=1}^n
|x_i|^{\lambda-2}v_i^2\geq 0,
$$
and $v^\top \nabla n^{\partial\uwidehat{C_\lambda}}(x) v=0$ if and only if $v=0$. Thus, the $(n-1)$ eigenvalues of the shape operator $v\in T_x\M\mapsto S_x(v)=\nabla n^{\partial\uwidehat{C_\lambda}}(x) v\in T_x\M$ are strictly positive. We conclude that the principal curvatures of $\M$ are strictly positive, so $\partial\uwidehat{C_\lambda}$ is an ovaloid.

The second claims follow from Lemma \ref{lem:sublevel_set_hausdoorff_to_zero}, noting that 
$$
\sup_{x\in\widehat{C_1}}|h_\lambda(x)-h(x)|\to 0\text{ as }\lambda\to\infty
$$
and where the second condition of Lemma \ref{lem:sublevel_set_hausdoorff_to_zero} holds for any $x\in\partial \uwidehat{C_\lambda}$ (or any $x\in\partial\widehat{C_\lambda}$) by taking $z=x-\epsilon n^{\partial C_\lambda}(x)$ for $\epsilon>0$ small-enough, since $n^{\partial C_\lambda}(x)$ is outward-pointing (this argument can be made rigorous using boundary charts, see e.g. \cite[Lemma 4.6]{LewBonalliJansonPavone2023}).
\phantom{a} \hfill $\blacksquare$

\subsection{Small differences in dynamics, initial conditions, and disturbances imply small reachable set errors}
\begin{lemma}[Continuity of the reachable sets]\label{lem:stability_reachset_errors}
Consider two sets of disturbances $\W^1,\W^2\subset\R^m$ and initial conditions $\X_0^1,\X_0^2\subset\R^n$, two dynamical systems parameterized by $f_1,f_2:\R\times\R^n\to\R^n$ and $g_1,g_2:\R\times\R^n\to\R^{n\times m}$, and defined by the ODEs
\begin{align}\label{eq:ODE_i}
\dot{x}(t)=f_i(t,x(t))+g_i(t,x(t))w(t),
\, t\in[0,T]
\end{align}
from $x(0)=x_0^i\in\X_0^i$ with $w\in L^\infty([0,T],\W^i)$ for $i=1,2$. 
Denote the solutions to \eqref{eq:ODE_i} by $x^i_{(w,x^0)}(\cdot)$ and the reachable sets for $t\in[0,T]$ by
\begin{align*}
\X_t^i=
\big\lbrace
x^i_{(w,x^0)}(t): 
w\in L^\infty([0,T],\W^i),
x^0\in\X_0^i
\big\rbrace.
\end{align*}
Assume that $((f_1,f_2),(g_1,g_2),(\W^1,\W^2),(\X_0^1,\X_0^2))$ satisfy Assumptions \ref{assumption:f}-\ref{assumption:X0}, respectively, 
and that for some $\epsilon>0$,
\begin{subequations}
\begin{align}
&\|f_1-f_2\|_\infty\leq\epsilon,
\quad
\|g_1-g_2\|_\infty\leq\epsilon,
\\
\label{eq:dH_W1W2_X01X02_epsilon}
&d_H(\X_0^1,\X_0^2)\leq\epsilon,\ \text{ and }\,
d_H(\W^1,\W^2)\leq\epsilon.
\end{align}
\end{subequations}
Then, for a finite constant $C_T$ and all $t\in[0,T]$, %
$$
d_H(\X_t^1,\X_t^2)\leq C_T\epsilon.
$$
\end{lemma}

\textit{Proof of Lemma \ref{lem:stability_reachset_errors}:}
For conciseness, we denote $x_t^i=x_{(w^i,x_0^i)}(t)$, $w_t^i=w^i(t)$, $f(x_t^i)=f(t,x_t^i)$, and $g(x_t^i)=g(t,x_t^i)$ for any $t\in[0,T]$ and $i=1,2$. 
It suffices to prove that given any $w^1\in L^\infty([0,T],\W^1)$ and $x_0^1\in\X_0^1$, there exists some $w^2\in L^\infty([0,T],\W^2)$ and $x_0^2\in\X_0^2$ such that $\sup_{t\in[0,T]}\|x_t^1-x_t^2\|\leq C_T\epsilon$.  
Given $w^1\in L^\infty([0,T],\W^1)$ and $x_0^1\in\X_0^1$, 
let $w^2\in L^\infty([0,T],\W_2)$ and $x_0^2\in\X_0^2$ be such that $\sup_{t\in[0,T]}\|w^1_t-w^2_t\|\leq\epsilon$ and $\|x_0^1-x_0^2\|\leq\epsilon$, which exist thanks to \eqref{eq:dH_W1W2_X01X02_epsilon}. First,

{\small
\begin{align}
\left\|\int_0^ta_s\dd s\right\|
&\leq
\left\|\int_0^ta_s\dd s\right\|_1
\leq
\int_0^t\left\|a_s\right\|_1\dd s
\leq
\sqrt{n}
\int_0^t\left\|a_s\right\|\dd s
\label{eq:ineq_integral_norm}
\end{align}
}%
for any integrable map $a:\R\to\R^n$. %
Then, 

{\small
\begin{align*}
\|x^1_t-x^2_t\|
&=
\bigg\|
x_0^1-x_0^2+\int_0^t (f_1(x^1_s)-f_2(x^2_s))\dd s 
\, +
\\
&\hspace{6mm}
\int_0^t (g_1(x^1_s)w^1_s-g_2(x^2_s)w^2_s)\dd s
\bigg\|
\\
&\hspace{-15mm}\leq
\left\|\Delta x_0\right\|+
\bigg\|
\int_0^t (f_1(x^1_s)-f_2(x^1_s)+f_2(x^1_s)-f_2(x^2_s))\dd s 
\, +
\\
&\hspace{-14mm}
\int_0^t ((g_1(x^1_s)-g_2(x^1_s)+g_2(x^1_s)-g_2(x^2_s))w^1_s+g_2(x^2_s)\Delta w_s)\dd s
\bigg\|
\\
&\hspace{-15mm}\mathop{\leq}^{\eqref{eq:ineq_integral_norm}}
\left\|\Delta x_0\right\|+
\sqrt{n}\int_0^t \|\Delta f(x^1_s)\|+L\|x^1_s-x^2_s\|\dd s 
\, +
\\
&\hspace{-15mm}
\sqrt{n}\int_0^t((\|\Delta g(x^1_s)\|+L\|x^1_s-x^2_s\|)\|w^1_s\|+\|g_2(x^2_s)\|\|\Delta w_s\|)\dd s
\\
&\hspace{-15mm}
\leq
\epsilon(1+\sqrt{n}(1+\bar{w}+\bar{g})t+
\sqrt{n}(L+L\bar{w})\int_0^t\|x_s^1-x_s^2\|\dd s,
\end{align*}
}%
where $\bar{w}=\sup_{w\in\W^1}\|w\|$, $L$ is the Lipschitz constant of $\max(f_2,g_2)$, and $\bar{g}=\sup_{(t,x)\in[0,T]\times \bar{\X}}\|g_2(t,x)\|$ for some set $\bar{\X}$ such that $\X_t\subseteq\bar{\X}$ for all $t\in[0,T]$  (note that the reachable sets $\X_t$ are compact by Lemma \ref{lem:Y_is_compact}, so $\bar{g}$ is finite). 
By Gr\"onwall's inequality, 
$$
\|x_t^1-x_t^2\|\leq
\epsilon(1+\sqrt{n}(1+\bar{w}+\bar{g})t
e^{\sqrt{n}(L+L\bar{w})}
=
C_T\epsilon.
$$
The conclusion follows. %
\hfill $\blacksquare$

\subsection{Details on the neural feedback loop analysis results}\label{appendix:numerical:nn}
We consider the system and neural network from \cite[Sec. VIII.A-C]{Everett21_journal}, replacing the \texttt{ReLU} activation functions with smooth \texttt{Softplus} activations so that Assumption \ref{assumption:f} is satisfied.\footnote{Note that by selecting appropriate hyperparameters, \texttt{Softplus} activations can be made arbitrarily close to \texttt{ReLU} activation functions.} We use the open-source implementation of \texttt{ReachLP} \cite{Everett21_journal}.  We discretize the dynamics with an Euler scheme at $\Delta t=0.25\textrm{s}$ and predict reachable sets over a horizon $T=4\textrm{s}$. 
The set of initial conditions is the ellipsoid $\X_0=\{x_0\in\R^2:(x_0-c_0)^\top Q_0^{-1}(x_0-c_0)\leq 1\}$ with $c_0=(2.75,0)$ and $Q_0=2\textrm{diag}(0.25^2,0.1^2)$. The set of disturbances is the ball $\W=B(0,\sqrt{2}/20)\subset\R^2$. 
\\
We use the sets of initial conditions $\bar{\X}_0=[2.5,3]\times[-0.1,0.1]$ and disturbances $\bar{\W}=\{w_t\in\R^2:\|w_t\|_\infty\leq 1/20\}$ for \texttt{ReachLP} (which only handles polytopic disturbance sets). Since $\bar{\X}_0\subset\X_0$ and $\bar{\W}\subset\W$, this choice for $(\bar{\X}_0,\bar{\W})$ is fair as it reduces the conservatism of \texttt{ReachLP}. We use the sets $\X_0$ and $\W$ for Algorithm \ref{alg:1} and \texttt{RandUP} \cite{LewJansonEtAl2022}, which satisfy Assumptions \ref{assumption:W} and \ref{assumption:X0}. 
The initial direction values $d^0$ for Algorithm \ref{alg:1} are selected to evenly cover the circle $\cS^{n-1}$.

\subsection{Details on the spacecraft attitude control results}\label{appendix:numerical:spacecraft}
\textit{Implementation}: We use $T=10\textrm{s}$, discretize \eqref{eq:spacecraft:dynamics} as %
\begin{equation}\label{eq:dynamics:discrete-time}
x((k+1)\Delta t)=\bar{f}(x(k\Delta t),\bar{u}(k\Delta t))+w(k\Delta t), 
\end{equation}
where $\bar{f}$ is given by a Runge-Kutta (RK4) scheme with $\Delta t=1$\,\textrm{s}, and enforce the constraints in \eqref{eq:spacecraft:constraints:approx} at each timestep $t_k=k\Delta t$, which is justified by the continuity of the state trajectories. 
We select $M=50$ samples $d^0\in\cS^{n-1}$ on a Fibonacci lattice \cite{Gonzalez2009}, which gives an internal $\delta$-covering of $\cS^{n-1}$ for $\delta\approx 0.5$. We evaluate $\bar{L}_t$ and $\bar{H}_t$ using $10^5$ samples\footnote{We evaluate the empirical bound $\bar{L}_t=\max_j\|\nabla F(d^j,t)\|$ %
using $10^5$ samples of the directions $d^j\in\cS^{n-1}$ %
and control inputs $\bar{u}^j$.} and directly use the error bounds $\epsilon_t$ predicted in \eqref{eq:error_bound}. This choice of $\epsilon_t$ ensures that the problem with the approximated constraints \eqref{eq:spacecraft:constraints:approx} gives solutions that satisfy the original constraints in \eqref{eq:spacecraft:constraints}, see also \cite{LewBonalliJansonPavone2023}.  %
We only parameterize the nominal state and control trajectories $(x_0,\bar{u})$ and evaluate \eqref{eq:spacecraft:constraints:approx} and its gradient as a function of $\bar{u}$. 
To solve the optimization problem, we use a standard sequential convex programming (SCP)  
\ifarxiv
    scheme \cite{LewBonalliEtAl2020,BonalliLewESAIM2022,BonalliLewEtAl2022} that
\fi 
consists of iteratively solving convex approximations of the problem until convergence. 
We use a Python implementation with \textrm{Jax} \cite{jax2018github} and solve the convexified problems using \textrm{OSQP} \cite{Stellato2020}.

\textit{Lipschitz-based reachability baseline}:
After discretizing \eqref{eq:spacecraft:dynamics} as  \eqref{eq:dynamics:discrete-time}, 
this standard baseline computes a reachable tube for the angular velocities $\mathcal{T}_t=\{y\in\R^3: (y-\omega_0(k\Delta t))^\top Q_{k}^{-1}(y-\omega_0(k\Delta t))\leq 1\}$ where $x_0=(q_0,\omega_0)$ follows nominal dynamics $x_0((k+1)\Delta t)=\bar{f}(x_0(k\Delta t),\bar{u}(k\Delta t))$ and the shape matrices $Q_k\in\R^{3\times 3}$ are recursively defined as $Q_0=0$, $Q_{k+1}=\frac{c+1}{c}Q^{\text{nom}}_k+(1+c)Q^{(\bar{w},\text{Lip})}_k$ for $k\in\bN$, 
where the first term $Q^{\text{nom}}_k=\bar{A}_k Q_k\bar{A}_k^\top$ with $\bar{A}_k=\nabla_x\bar{f}(x_0(k\Delta t), \bar{u}(k\Delta t))$ propagates uncertainty at time $k\Delta t$ with a linearized model, 
$Q^{(\bar{w},\text{Lip})}_k=3(\bar{w}+\frac{\bar{H}}{2}\lambda_{\max}(Q_k))^2I_3$ with $\bar{w}=10^{-2}$ and $\bar{H}$ the Lipschitz constant of $\nabla_x\bar{f}$ accounts for the disturbance and the linearization error, and $c^2=\text{Trace}(Q^{\text{nom}}_k)/\text{Trace}(\bar{Q}^{(\bar{w},\text{Lip})}_k)$.  This standard baseline is described in \cite{koller2018} (see also \cite{LewPavone2020} and \cite{Leeman2023}) and ensures that the trajectories resulting from the discrete-time dynamics \eqref{eq:dynamics:discrete-time} with $w(k\Delta t)\in\W$ are contained in the tube $\mathcal{T}$. 

\newpage

\begin{figure}[t]
\centering
\includegraphics[width=0.8\linewidth,trim=10 40 5 5, clip]{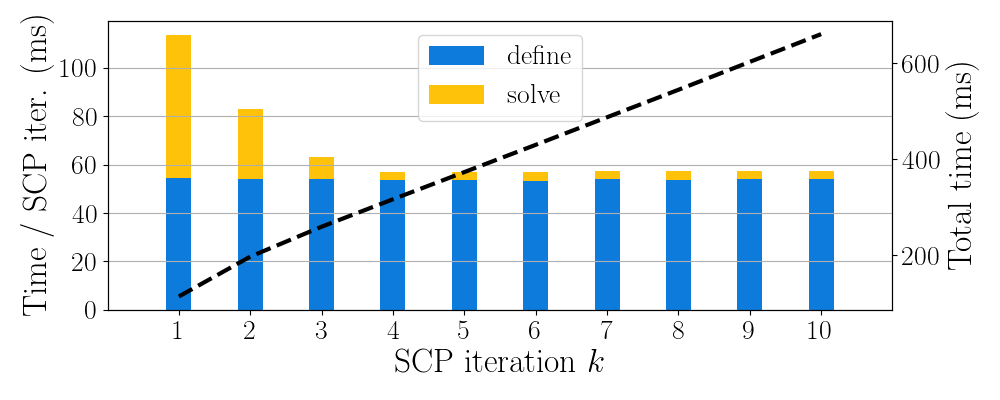}
\\
\includegraphics[width=0.8\linewidth,trim=10 20 5 5, clip]{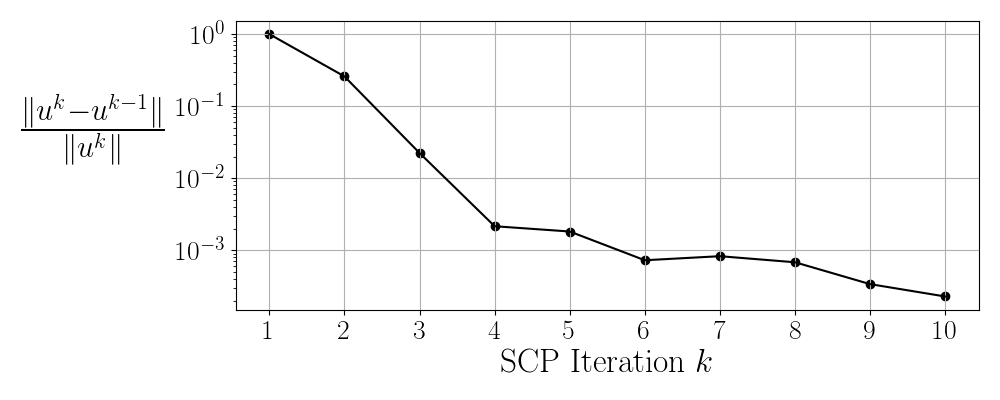}
\caption{Solving \OCPMPC{x^0} via SCP: median numerical resolution statistics over $50$ experiments for $M = 50$ samples with computation times (top) and SCP iteration errors (bottom).}
\label{fig:spacecraft:comp_times}
\includegraphics[width=0.6\linewidth,trim=0 30 0 0, clip]{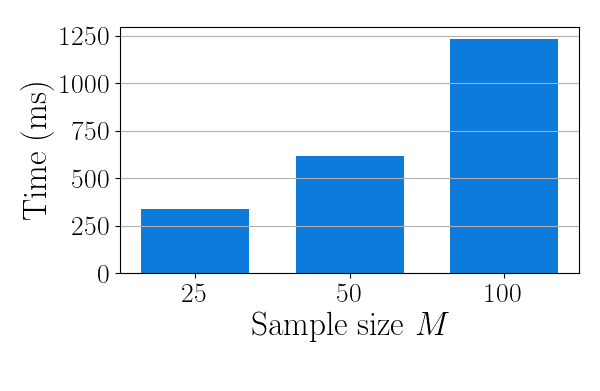}
\caption{Total computation time for solving \OCPMPC{x^0} for different number of samples $M$: median statistics over $50$ runs.}
\label{fig:spacecraft:comp_times_Ms}
\end{figure}

\phantom{a}
\vfill
\phantom{a}

\end{document}